\documentclass[a4paper]{amsart}
\usepackage{geometry}
\usepackage[foot]{amsaddr}
\usepackage{amsthm,amssymb}

\usepackage{tikz}
\usetikzlibrary{shapes,arrows,positioning}
\usepackage[utf8]{inputenc}
\usepackage{amsmath,amssymb,amsthm,commath,esint,tikz-cd,tikz, mathtools, physics}
\usepackage[mathscr]{euscript}
\usepackage{amsfonts}
\usepackage{subcaption}
\usepackage{hyperref}
\usepackage{graphicx}
\usepackage{float}
\usepackage{fancyhdr}
\usepackage{bm}
\usepackage{bbm}
\usepackage{booktabs}
\usepackage[ruled,vlined]{algorithm2e}
\DeclareMathOperator*{\esssup}{ess\,sup}

\numberwithin{equation}{section}

\usepackage{amsthm,amssymb}
\usepackage{mathtools,thmtools}
\usepackage{bm,esint}
\usepackage{color}
\usepackage{float,graphicx,subcaption}
\usepackage{cancel}
\usepackage[shortlabels]{enumitem}
\usepackage{mathrsfs}  
\usepackage{bbm}
\usepackage{euscript}
\usepackage{hyperref}
\usepackage{cleveref}
\usepackage{upgreek}

\usepackage{todonotes}

\renewcommand{\hat}{\widehat}

\DeclareMathOperator*{\essinf}{ess\,inf}
\DeclareMathOperator*{\argmin}{argmin}

\newcommand{\R}{\mathbb{R}}

\newcommand{\N}{\mathbb{N}}

\renewcommand{\S}{\mathscr{S}}
\newcommand{\cA}{\mathcal{A}}

\newcommand{\cC}{\mathcal{C}}

\newcommand{\cE}{\mathcal{E}}

\newcommand{\cL}{\mathcal{L}}

\newcommand{\cR}{\mathcal{R}}

\newcommand{\cT}{\mathcal{T}}

\newcommand{\uhat}{{\widehat{u}}}

\renewcommand{\hat}{\widehat}

\renewcommand{\epsilon}{\varepsilon}

\renewcommand{\S}{\mathcal{S}} 
\newcommand{\Et}{\mathcal{E}_T} 
\newcommand{\Eg}{\mathcal{E}_G} 

\newcommand{\bigO}{\mathcal{O}}
\newcommand{\Prob}[1]{{\mathbb{P}\left( #1 \right)}} 

\newcommand{\sgn}[1]{\mathrm{sgn}\left( #1 \right)}

\newtheorem{theorem}{Theorem}[section]

\newtheorem{remark}[theorem]{Remark}
\newtheorem{definition}[theorem]{Definition}
\newtheorem{lemma}[theorem]{Lemma}

\newtheorem{corollary}[theorem]{Corollary}

\title[wPINNs for approximating conservation laws]{wPINNs: \emph{Weak} Physics informed neural networks for approximating entropy solutions \\of hyperbolic conservation laws}

\author{T.~De Ryck}

\author{S.~Mishra}

\author{R.~Molinaro}
\address[T. De Ryck, S. Mishra and R. Molinaro]{Seminar for Applied Mathematics, ETH Z\"urich, R\"amistrasse 101, 8092 Z\"urich, Switzerland}

\begin{document}

\maketitle


\begin{abstract}
Physics informed neural networks (PINNs) require regularity of solutions of the underlying PDE to guarantee accurate approximation. Consequently, they may fail at approximating discontinuous solutions of PDEs such as nonlinear hyperbolic equations. To ameliorate this, we propose a novel variant of PINNs, termed as weak PINNs (\emph{wPINNs}) for accurate approximation of entropy solutions of scalar conservation laws. \emph{wPINNs} are based on approximating the solution of a min-max optimization problem for a residual, defined in terms of Kruzkhov entropies, to determine parameters for the neural networks approximating the entropy solution as well as test functions. We prove rigorous bounds on the error incurred by \emph{wPINNs} and illustrate their performance through numerical experiments to demonstrate that \emph{wPINNs} can approximate entropy solutions accurately.       
\end{abstract}

\section{Introduction}
Driven by their well-documented successes in fields ranging from computer vision and natural language understanding to robotics and autonomous vehicles, machine learning techniques, particularly deep learning \cite{DLnat}, are being increasingly used in scientific computing. Given the fact that \emph{deep neural networks} (DNNs) are universal function approximators, they have proved to be a popular choice for ansatz spaces for the construction of fast surrogates for approximating a variety of partial differential equations (PDEs) including elliptic \cite{SZ1,Kuty}, parabolic \cite{HEJ1} and hyperbolic PDEs \cite{LMR1,LMPR1}. Even, the underlying solution operators can be \emph{learned} using DNN based operator learning frameworks such as DeepONets \cite{deeponets} or Fourier Neural Operators \cite{FNO}. 

All the afore-mentioned approaches fall under the rubric of \emph{supervised learning} and require significant amounts of data for training the underlying DNNs. However, this training data is generated from either observations (experiments) or numerical simulations and can be very expensive to access and store. Hence, in many contexts, one needs a learning framework that can approximate the underlying PDE solutions with very little or possibly \emph{no data}. Physics informed neural networks (PINNs) provide a very popular example of such an \emph{unsupervised learning} framework. In contrast to supervised learning approaches where the mismatch between the ground truth and neural network predictions is minimized during training, training of PINNs relies on the minimization of an underlying (pointwise) residual associated with the PDE in suitable hypotheses spaces of neural networks. 

PINNs were first proposed, albeit in slightly different form, in \cite{DPT,Lag1,Lag2} but they were resurrected and popularized more recently in \cite{KAR1,KAR2} as an efficient alternative to traditional numerical methods, for solving both forward as well as inverse problems for PDEs. Since their reintroduction, the use of PINNs has grown exponentially in different areas of scientific computing and a very selected list of references include \cite{KAR4,KAR6,KAR7,KAR8,jagtap2020extended,jag2,jagtap2022physics, jin2021nsfnets,MM1,MM2,MM3,BKMM1, shukla2021physics, jagtap2022deep, hu2021extended, shukla2021parallel,MM1,MM2,MM3,deryck2021navierstokes,deryck2021approximation} and references therein, see also \cite{cuomo} for an extensive recent review of PINNs. 

Although less advanced than the applications of PINNs in various domains, there has been significant development of the theory for PINNs recently, particularly in the form of rigorous bounds on various sources of the underlying error. These include \cite{shin2020convergence} where the authors show consistency of PINNs with the underlying linear elliptic and parabolic PDE under stringent assumptions and in \cite{Zhang1} where similar estimates are derived for linear advection equations. In \cite{MM1,MM2}, the authors provided a roadmap for deriving error estimates for PINNs and applied it for both the forward and inverse problems in a variety of elliptic, parabolic and convection-dominated PDEs. This strategy was further refined and adapted to very-high dimensional Kolmogorov PDEs in \cite{deryck2021pinn} as well to the Navier-Stokes equations in \cite{deryck2021navierstokes}, among others. The key elements of this strategy, as also identified in \cite{deryck2022generic}, are as follows,
\begin{itemize}
    \item [i.)] \emph{Regularity} of the solutions of the underlying PDEs, which enables one to apply estimates on the error (in high-enough Sobolev norms), incurred by neural networks in approximating smooth functions. This regularity is leveraged into proving that the PDE residuals, which are to be minimized during the training process, can be made arbitrarily small.
     \item [ii.)] \emph{Coercivity} (or stability) of the underlying PDEs which ensure that the total error can be estimated in terms of the residuals.  For nonlinear PDEs, the constants in these coercivity estimates often depend on the regularity of the underlying solutuions.
     \item [iii.)] \emph{Quadrature error} bounds for estimating the so-called generalization gap between the continuous and discrete versions of the PDE residual \cite{deryck2021approximation}. 
\end{itemize}
These elements also bring forth the \emph{limitations of PINNs} by highlighting PDEs where PINNs might not provide an accurate approximation. In particular, there are a large number of contexts in which solutions of PDEs might not be smooth, even for smooth inputs. The analysis of \cite{MM1,deryck2021navierstokes,deryck2022generic} suggests that conventional forms of PINNs will fail to accurately approximate solutions of these PDEs. 

A prototypical example of a class of PDEs for which the underlying solutions are not smooth, is provided by nonlinear hyperbolic systems of conservation laws such as the Euler, shallow-water and ideal Magneto-Hydrodynamics (MHD) equations \cite{holden2015front}. Even in the simplest case of a \emph{scalar conservation law}, e.g. inviscid Burgers' equation, it is well-known that, even if the initial datum is smooth, discontinuities in the form of \emph{shock waves} form within a finite time. Hence, the underlying equation can no longer be interpreted in a (pointwise) \emph{strong} sense, rather \emph{weak solutions} need to be considered \cite{holden2015front}. Moreover, these weak solutions are no longer unique. Additional admissibility criteria or \emph{entropy conditions} have to be imposed in order to restore uniqueness \cite{holden2015front}. 

As the solutions of hyperbolic conservation laws can be discontinuous, the analysis of \cite{MM1,deryck2022generic} and references therein suggests that PINNs cannot accurately approximating the underlying weak solutions of these PDEs. This fact is also empirically verified in \cite{MM1} (Figure 6 (d)) where unacceptably large errors were observed for PINNs approximating scalar conservation laws. This can be explained by the fact that the pointwise residual blows up for smooth approximations of the underlying exact solution and minimizing it in the class of neural networks is futile.  

Given this context, we ask if one can modify PINNs to design an \emph{unsupervised learning framework} for approximating (entropy) weak solutions of hyperbolic conservation laws and related equations, such that the resulting error can be rigorously proved to be arbitrarily small. Answering this question affirmatively constitutes the central rationale of the current paper. 

To this end, we will focus on the simple yet prototypical case of scalar conservation laws here. As mentioned earlier, the pointwise residuals associated with (approximations of) weak solutions can blow up. Hence, we need to replace these pointwise PDE residuals with suitable \emph{weak} versions. This can be achieved by mimicking the weak formulation of the underlying conservation laws \cite{holden2015front} and integrating by parts with respect to smooth \emph{test functions} to define a weak form of the PDE residual. Such weak versions of PINNs have already been considered in the context of so-called \emph{variational PINNs} \cite{vpinns}, where the test functions are selected as suitable basis functions, such as orthogonal polynomials. Such an approach can indeed be considered in our context. However, we refrain from doing so here as a key advantage for PINNs is that it is a \emph{meshless} approach and does not require any underlying grid. However, variational PINNs are often defined in terms of locally supported functions on a mesh. Instead, we will leverage the universal approximation properties of neural networks and choose parametrized neural networks as our test functions. Moreover, neural networks are also used as approximations to the underlying solution i.e., as trial functions. Hence, our weak formulation leads to a \emph{min-max} optimization problem where the neural network parameters (weights and biases) are maximized with respect to test functions and minimized with respect to trial functions. Such min-max problems arise in machine learning when training generative adversarial networks or GANs, \cite{gan,wgan} and references therein. 

However, working with the weak formulation alone does not suffice to build an accurate approximation strategy for scalar conservation laws as one also needs to incorporate entropy conditions. To this end, we will define a novel \emph{entropy residual}, based on the well-known family of \emph{Kruzkhov entropies} \cite{holden2015front} and solve the corresponding min-max optimization problem for training the neural networks that will approximate the entropy solution accurately. We term the resulting construction as \emph{wPINNs} or weak PINNs and prove rigorous error bounds on them. In particular, we will prove that the entropy residuals, associated with \emph{wPINNs} can be made arbitrarily small. We rely on Kruzkhov's method of \emph{doubling of variables} to prove that the error of the \emph{wPINN} with respect to the entropy solution can be bounded in terms of the (continuous version) of the entropy residual. Finally, statistical learning theory arguments, already considered in the context of PINNs in \cite{deryck2021approximation}, are adapted to yield a bound on the generalization gap. Thus, taken together, these bounds provide rigorous estimates on the error for \emph{wPINNs} and prove that \emph{wPINNs} can approximate entropy solutions of scalar conservation laws accurately. We also provide numerical experiments to illustrate the efficient approximation of entropy solutions by \emph{wPINNs}. Hence, we design a novel variant of PINNs to approximate discontinuous entropy solutions of conservation laws and prove that the approximation is accurate. 
\section{\emph{wPINNs} for Scalar conservation laws}
\label{sec:2}
In this section, we will describe the construction of \emph{wPINNs} for approximating the entropy solutions of scalar conservation laws. We start by recapitulating some basic concepts about conservation laws. 
\subsection{Scalar Conservation Laws}
\label{sec:21}
For simplicity of exposition, we will focus on the case of one spatial dimension while mentioning that extending the results to several space dimensions is straightforward. Without loss of generality, we fix $D = [0,1] \subset \R$ as the spatial domain and consider the scalar conservation law, 
\begin{equation}
\label{eq:scl}
\begin{aligned}
    u_t + f(u)_x  = 0&\quad \text{in}~  D\times [0,T]\\
        u = u_0&\quad \text{on}~  D\times \{0\}.
\end{aligned}
\end{equation}
Here, $u \in L^1(D \times (0,T))$ is the conserved quantity and $f$ is the so-called flux function with $u_0$ being the initial data. Moreover, the PDE \eqref{eq:scl} needs to be supplemented with suitable boundary conditions. We mostly consider periodic boundary conditions in this paper.

Following \cite{holden2015front}, one defines \emph{weak solutions} of \eqref{eq:scl} as follows,
\begin{definition}\label{def:weak-solution}
A function $u\in L^\infty(\R\times \R_+)$ is a weak solution of \eqref{eq:scl} with initial data $u_0\in L^\infty(\R)$ if
\begin{equation}\label{eq:weak-solution}
    \int_{\R_+}\int_\R \left(u\varphi_t + f(u)\varphi_x\right) dxdt + \int_\R u_0(x)\varphi(x,0)dx = 0 ,
\end{equation}
holds for all test functions $\varphi\in C^1_c(\R\times \R_+)$. 
\end{definition}
However, weak solutions are not unique \cite{holden2015front}. To recover uniqueness, one needs to impose additional admissibility criteria or \emph{entropy conditions}. To this end, we consider the so-called \emph{Kruzkhov entropy functions}, given by $|u-c|$, for any $c\in\R$ and the resulting entropy flux functions,
\begin{equation}\label{eq:def-q}
    \partial_t \abs{u-c}+\partial_x Q[u;c]\leq 0 \qquad \text{where}\quad Q:\R^2\to \R: (u,c)\mapsto \sgn{u-c}(f(u)-f(c)).
\end{equation}
With this notation, we have the following definition of \emph{entropy solutions},
\begin{definition}\label{def:entropy}
We say that a function $u\in L^\infty(\R\times \R_+)$ is an entropy solution of \eqref{eq:scl} with initial data $u_0\in L^\infty(\R)$ if $u$ is a weak solution of \eqref{eq:scl} and if $u$ satisfies that 
\begin{equation}\label{eq:kruzkov}
    \int_{0}^T\int_\R \left(\abs{u-c}\varphi_t + Q[u;c]\varphi_x\right) dxdt - \int_\R\left(\abs{u(x,T)-c}\varphi(x,T)- \abs{u_0(x)-c}\varphi(x,0)\right)dx \geq 0
\end{equation}
for all $\varphi\in C^1_c(\R\times \R_+)$, $c\in \R$ and $T>0$. 
\end{definition}
It holds that these entropy solutions are unique and continuous in time, as formulated below \cite{holden2015front}, where $\norm{\cdot}_{TV}$ denotes the total variation seminorm. 
\begin{theorem}\label{thm:entropy-sol}
Assume that $f\in C^1$ and $u_0\in L^\infty\cap L^1$. Then there exists a unique entropy solution $u$ of \eqref{eq:scl} and if $\norm{u_0}_{TV}<\infty$ then $u$ satisfies the following, 
\begin{equation}
    \norm{u(t)-u(s)}_{L^1}\leq \abs{t-s}M\norm{u_0}_{TV}\qquad \text{and}\qquad \norm{u(t)}_{L^\infty}\leq \norm{u_0}_{L^\infty}, ~ \norm{u(t)}_{BV}\leq \norm{u_0}_{BV},
\end{equation}
where $M = M(u_0) = \max_{\essinf_x u_0(x)\leq u\leq \esssup_x u_0(x)} \abs{f'(u)}$. 
\end{theorem}

\subsection{Neural networks}
\label{sec:22}
Our aim in this paper is to approximate entropy solutions of \eqref{eq:scl} with neural networks, which we formally define next. 
\begin{definition}
\label{def:nn-app}
Let $R\in(0,\infty]$, $L,W\in\mathbb{N}$ and $l_0,\ldots, l_L\in\mathbb{N}$. Let $\sigma:\mathbb{R}\to\mathbb{R}$ be a measurable function, the so-called \emph{activation function}, and define 
\begin{equation}
   \Theta =  \Theta_{L,W,R} := \bigcup_{L'\in\mathbb{N}, L'\leq L}\:\bigcup_{l_0,\ldots,l_L\in\{1, \ldots, W\}}\bigtimes_{k=1}^{L'} \left([-R,R]^{l_k\times l_{k-1}}\times[-R,R]^{l_k}\right). 
\end{equation}
For $\theta\in\Theta_{L,W,R}$, we define $(W_k,b_k):=\theta_k$ and $\mathcal{A}_k^\theta:\mathbb{R}^{l_{k-1}}\to\mathbb{R}^{l_{k}}:z\mapsto W_k z+b_k$ for $1\leq k\leq L$ and we denote by $u_\theta:\mathbb{R}^{l_0}\to\mathbb{R}^{l_L}$ the function that satisfies for all $z \in\mathbb{R}^{l_0}$ that
\begin{equation}
\label{eq:dnn}
     u_\theta(z) = \left(\cA_{L}^\theta\circ\sigma\circ \cA_{L-1}^\theta \circ \cdots \circ\sigma\circ \cA_1^\theta\right)(z), 
\end{equation}
where in the setting of approximating solutions to PDEs we set $l_0 = d+1$ and $z=(x,t)$.

We refer to $u_\theta$ as the realization of the neural network associated to the parameter $\theta$ with $L$ layers with widths $(l_0,l_1, \ldots, l_L)$, of which the middle $L-1$ layers are called hidden layers. For $1\leq k\leq L$, we say that layer $k$ has width $l_k$ i.e., we say that it consists of $l_k$ neurons, and we refer to $W_k$ and $b_k$ as the weights and biases corresponding to layer $k$. If $L\geq 3$, we say that $u_\theta$ is a deep neural network (DNN). The total number of neurons in the network is given by the sum of the layer widths, $\sum_{k=0}^L l_k$. Note that the weights and biases of neural network $u_\theta$ with $\theta \in \Theta_{L,W,R}$ are bounded by $R$. 

\end{definition}
\subsection{Physics informed neural networks (PINNs)}
\label{sec:23}
We briefly introduce PINNs and how these neural networks could be used to approximate the solutions of the scalar conservation law \eqref{eq:scl}. Given that we are on a finite domain, we (formally) write down the boundary conditions supplementing \eqref{eq:scl} as, $u = g$ on $\partial D \times [0,T]$. Then, one can consider the following (pointwise) residuals associated with the PDE \eqref{eq:scl} and initial (and boundary) conditions,
\begin{align}\label{eq:pinn-residuals}
    \begin{split}
        r_{int}[v](x,t) = v_t+f(v)_x, \qquad r_{sb}[v](y,t) = v(y,t)-g(y,t), \qquad r_{tb}[v](x) =v(x,0) - u_0(x),
    \end{split}
\end{align}
where $v\in C^1(D\times [0,T])$ and where $r_{int}$ is the PDE residual, $r_{sb}$ is the (spatial) boundary residual and $r_{tb}$ is the temporal boundary residual stemming from the initial condition. Ideally, one would then solve the following minimization problem,
\begin{equation}\label{eq:pinn-min}
    \widehat{\theta} := \argmin_{\theta\in\Theta}\left(\norm{r_{int}[u_\theta]}^2_{L^2(D\times [0,T])} + \lambda_{sb}\norm{r_{sb}[u_\theta]}^2_{L^2(\partial D\times [0,T])} + \lambda_{tb}\norm{r_{tb}[u_\theta]}^2_{L^2(D)}\right),
\end{equation}
and define the PINN as $u_{\widehat{\theta}}$, with $\theta$ in \eqref{eq:pinn-min} denoting tunable parameters (weights and biases) of the underlying neural networks \eqref{eq:dnn}. In practice, one cannot exactly solve the above minimization problem. Instead, one approximates the integrals in \eqref{eq:pinn-min} using a numerical quadrature and one tries to find an approximate minimizer using a (stochastic) gradient descent algorithm. 

One can already observe from the form of the PDE residual in \eqref{eq:pinn-residuals} that it is imposed pointwise. Although well-defined as long as the activation function $\sigma$ in \eqref{eq:dnn} is smooth, one expects this residual to blow up as the residual corresponding to smooth approximations of entropy solutions of conservation laws, such as viscous profiles, blows up as the profile width vanishes \cite{holden2015front,MM1}. Hence, we cannot expect that PINNs will approximate weak solutions of \eqref{eq:scl} accurately, particularly near shocks. This is indeed observed empirically in \cite{MM1} (Figure 6 (d)). 

\subsection{Weak PINNs (\emph{wPINNs})}
\label{sec:24}
As mentioned in the introduction, we will circumvent the failure of conventional PINNs that minimized the pointwise PDE residual \eqref{eq:pinn-min} by searching for neural networks that minimize a residual, related to the Kruzkhov entropy condition instead. To this end,  we define for $v\in (L^\infty\cap L^1)(D\times [0,T])$, $\varphi \in W^{1,\infty}_0(D\times [0,T])$ and $c\in\R$ the following \emph{Kruzhov entropy residual},
\begin{equation}\label{eq:R-def}
    \cR(v,\varphi, c) := - \int_D\int_{[0,T]} \left(\abs{v(x,t)-c}\partial_t\varphi(x,t) + Q[v(x,t); c]\partial_x\varphi(x,t) \right)dx dt.
\end{equation}
Note that if $u$ is an entropy solution of \eqref{eq:scl}, then it holds that $\cR(u,\varphi, c)\leq 0$. This suggests solving the following min-max counterpart to the PINN minimization problem \eqref{eq:pinn-min}, 
\begin{equation}\label{eq:minimax}
    \theta^* := \argmin_{\theta\in\Theta}\max_{\varphi\in W^{1,\infty}_0(D\times [0,T]), c\in\R}\left(\cR(u_\theta,\varphi, c) +  \lambda_{sb}\norm{r_{sb}[u_\theta]}^2_{L^2(\partial D\times [0,T])} + \lambda_{tb}\norm{r_{tb}[u_\theta]}^2_{L^2(D)}\right),
\end{equation}  
The DNN associated with the $u_{\theta^*}$ is defined as the weak PINN, \emph{wPINN} for short, and we will show that it provides an accurate approximation of the entropy solution $u$ of \eqref{eq:scl}. 

We observe that in practice, the test function space $W^{1,\infty}_0$ needs to be replaced by a finite-dimensional approximation. One possibility is to use locally supported (piecewise) polynomials or orthogonal polynomials such as Legendre polynomials. Such choices lead to what is often termed as variational PINNs. However, we will depart from this choice and leverage the universal approximation property of neural networks to approximate  $W^{1,\infty}_0$ by parametrized neural networks of the form \eqref{eq:dnn} with a $C^1$-activation function. This leads to a min-max optimization problem where maximum as well as the minimum are taken with respect to neural network training parameters. We refer to Section \ref{sec:weak-pinn-training} for details on how the min-max problem \eqref{eq:minimax} can be solved in practice. 
\section{Error Analysis for \emph{wPINNs}}
\label{sec:3}
In this section, we will provide rigorous bounds on different components of the error incurred by \emph{wPINNs} in approximating the entropy solution of the conservation law \eqref{eq:scl}. 
\subsection{Bounds on the Entropy Residual}
\label{sec:approximation}
We start the rigorous analysis of \emph{wPINNs} by following the layout of \cite{deryck2021approximation} and asking if the entropy residual \eqref{eq:R-def} can be made arbitrarily small, within the class of neural networks. An affirmative answer to this question will confirm that the choice of the loss function for \emph{wPINN} is suitable. To investigate this question, we start with the following observation, 
\begin{lemma}\label{lem:Q}
Let $f:\R\to\R$ be Lipschitz with constant $L_f$, let $Q$ be as in \eqref{eq:def-q} and let $c,u,v\in\R$. It holds that $\abs{Q[u; c]-Q[v; c]}\leq 3L_f\abs{u-v}$. 
\end{lemma}
\begin{proof}
We calculate,
\begin{equation}
    \begin{split}
Q[u; c]-Q[v; c] &= \sgn{u-c}(f(u)-f(c))-\sgn{v-c}(f(v)-f(c))\\
&= (\sgn{u-c}-\sgn{v-c})(f(u)-f(c)) + \sgn{v-c}(f(u)-f(v))\\
&=: T_1 + T_2.
    \end{split}
\end{equation}
Note that if $\sgn{u-c}\neq\sgn{v-c}$ then either $u<c<v$ or $v<c<u$ pointwise almost everywhere. Consequently, if $\sgn{u-c}\neq\sgn{v-c}$ then $\abs{u-c}\leq \abs{u-v}$. 
We therefore find that
\begin{equation}
    \abs{T_1} \leq L_f \abs{\sgn{u-c}-\sgn{v-c}}\abs{u-c} \leq 2L_f \abs{u-v}. 
\end{equation}
Moreover, it holds that $\abs{T_2} \leq L_f \abs{u-v}$. In total we therefore have that $\abs{Q[u; c]-Q[v; c]}\leq 3L_f\abs{u-v}$. 
\end{proof}

Next, we prove that for any test function $\varphi\in W_0^{1,\infty}$, any $q\geq 1$ and any $c\in\R$ the quantity $\cR(v,\varphi, c)$ can be bounded in terms of the $L^q$-norm of $v-u$.

\begin{theorem}\label{thm:approx-general}
Let $p,q>1$ be such that $\frac{1}{p}+\frac{1}{q}=1$ or let $p=\infty$ and $q=1$. Let $u$ be the entropy solution of \eqref{eq:scl} and let $v\in L^q(D\times [0,T])$. Assume that $f$ has Lipschitz constant $L_f$. Then it holds for any $\varphi\in W_0^{1,\infty}(D\times [0,T])$ that
\begin{equation}
    \cR(v,\varphi, c) \leq (1+3L_f)\abs{\varphi}_{W^{1,p}(D\times [0,T])} \norm{u-v}_{L^q(D\times[0,T])}. 
\end{equation}
\end{theorem}

\begin{proof}
Let $c$ and $\varphi$ be arbitrary. In the following, we will write $L^p$ for $L^p(D\times[0,T])$. 
\begin{equation}
\begin{split}
\cR(v,\varphi, c) & = \cR(v,\varphi, c) - \cR(u,\varphi, c) + \cR(u,\varphi, c) \\
&\leq \cR(v,\varphi, c) - \cR(u,\varphi, c)\\
& = \int_D\int_{[0,T]} \left(\abs{u(x,t)-c}-\abs{v(x,t)-c}\right)\partial_t\varphi(x,t)+\left(Q[u(x,t); c]-Q[v(x,t); c]\right)\partial_x\varphi(x,t) dx dt\\
&\leq \norm{\partial_t\varphi}_{L^p} \norm{u-v}_{L^q}  + \norm{\partial_x\varphi}_{L^p} \norm{Q[u; c]-Q[v; c]}_{L^q}
\end{split}
\end{equation}
Using Lemma \ref{lem:Q} we find that $\norm{Q[u; c]-Q[v; c]}_{L^q} \leq 3L_f \norm{u-v}_{L^q}$ and therefore,
\begin{equation}
    \cR(v,\varphi, c) \leq  (1+3L_f)\abs{\varphi}_{W^{1,p}(D\times [0,T])} \norm{u-v}_{L^q(D\times[0,T])}. 
\end{equation}
\end{proof}

The above result implies that we have to show that the entropy solution of \eqref{eq:scl} can be approximated by neural networks \eqref{eq:dnn} for the entropy residual \eqref{eq:R-def} to be sufficiently small, within the class of neural networks. In the following, we focus on the hyperbolic tangent (tanh) activation function, 
\begin{equation}
    \sigma(x) := \tanh(x) = \frac{\exp(x)-\exp(-x)}{\exp(x)+\exp(-x)},
\end{equation}
for the sake of definiteness, while observing that extensions to other smooth activation functions can be similarly considered. We have the following result on the approximation of $BV$-functions by tanh neural networks, 
\begin{lemma}\label{lem:bv-nn}
For every $u\in BV([0,1]\times [0,T])$ and $\epsilon>0$ there is a tanh neural network $\uhat$ with two hidden layers and at most $O(\epsilon^{-2})$ neurons such that
\begin{equation}
    \norm{u-\uhat}_{L^1([0,1]\times [0,T])}\leq \epsilon.
\end{equation}
\end{lemma}
\begin{proof}
Write $\Omega = [0,1]\times [0,T]$. For every $u\in BV(\Omega)$ and $\epsilon>0$ there exists a function $\Tilde{u}\in C^\infty(\Omega)\cap BV(\Omega)$ such that $\norm{u-\Tilde{u}}_{L^1(\Omega)}\lesssim \epsilon$ and $\norm{\nabla \Tilde{u}}_{L^1(\Omega)} \lesssim \norm{u}_{BV(\Omega)}+\epsilon$ \cite{bartels2012total}. 
Then use the approximation techniques of \cite{deryck2021approximation, deryck2021navierstokes} and the fact that $\norm{\Tilde{u}}_{W^{1,1}(\Omega)}$ can be uniformly bounded in $\epsilon$ to find the existence of a tanh neural network $\uhat$ with two hidden layers and at most $O(\epsilon^{-2})$ neurons that satisfies the mentioned error estimate. 
\end{proof}

If one additionally knows that $u$ is piecewise smooth, for instance as in the solutions of convex scalar conservation laws \cite{holden2015front}, then one can use the results of \cite{petersen2018optimal} to obtain the following result. 

\begin{lemma}\label{lem:pw-nn}
Let $m,n\in\N$, $1\leq q<\infty$ and let $u:[0,1]\times [0,T]\to \R$ be a function that is piecewise $C^m$ smooth and with a $C^n$ smooth discontinuity surface. Then there is a tanh neural network $\uhat$ with at most three hidden layers and $\bigO(\epsilon^{-2/m}+\epsilon^{-q/n})$ neurons such that
\begin{equation}
    \norm{u-\uhat}_{L^q([0,1]\times [0,T])}\leq \epsilon.
\end{equation}
\end{lemma}
\begin{proof}
The proof follows the lines of \cite[Appendix A]{petersen2018optimal} with a few adaptations. We approximate the Heaviside function by the function $x\mapsto \frac{1}{2}(\sigma(\beta x)+1)$, where $\beta = \frac{1}{2\epsilon}\ln(1/\epsilon)$ for some $\epsilon>0$. 
This replaces \cite[Lemma A.2]{petersen2018optimal}. 
Moreover, we replace \cite[Lemma A.3]{petersen2018optimal}, which discusses approximation rates for the multiplication operator, by \cite[Corollary 3.7]{deryck2021approximation} and we replace \cite[Lemma A.9]{petersen2018optimal}, which discusses approximation rates for smooth functions, by \cite[Theorem 5.1]{deryck2021approximation}. 
\end{proof}
Finally, from Lemma \ref{lem:Phi}, stated and proved in the Appendix, we find that one has to consider test functions $\varphi$ that might grow as $\abs{\varphi}_{W^{1,p}}\sim\epsilon^{-3(1+2(p-1)/p)}$. Consequently, we will need to use Lemma \ref{lem:bv-nn} with $\epsilon \leftarrow \epsilon^{-4+6(p-1)/p}$, leading to the following corollary.

\begin{corollary}
\label{cor:er}
Assume the setting of Lemma \ref{lem:pw-nn}, assume that $m q > 2n$ and let $p\in[1,\infty]$ be such that $\frac{1}{p}+\frac{1}{q}=1$. There is a fixed depth tanh neural network $\uhat$ with size $\bigO(\epsilon^{-(4q+6)n})$ such that
\begin{equation}
    \max_{c\in \cC}\sup_{\varphi\in\overline{\Phi}_\epsilon}\cR(\uhat,\varphi, c)+  \lambda_{sb}\norm{r_{sb}[\uhat]}^2_{L^2(\partial D\times [0,T])} + \lambda_{tb}\norm{r_{tb}[\uhat]}^2_{L^2(D)} \leq \epsilon, 
\end{equation}
where $\overline{\Phi}_\epsilon = \{\varphi\: :\: \abs{\varphi}_{W^{1,p}}=\bigO(\epsilon^{-3(1+2(p-1)/p))}\}$.
\end{corollary}
\begin{proof}
We find that $\uhat$ will need to have a size of $\bigO(\epsilon^{-(4q+6)/\beta})$ such that $\max_{c\in \cC}\sup_{\varphi\in\Phi_\epsilon}\cR(\uhat,\varphi, c) \leq \epsilon/3$, where we used that $q = p/(p-1)$. Since in the proof of Lemma \ref{lem:pw-nn} the network $\uhat$ is constructed as an approximation of piecewise Taylor polynomials, the spatial and temporal boundary residuals ($r_{sb}$ and $r_{tb}$) are automatically minimized as well, given that Taylor polynomials provide approximations in $C^0$-norm. 
\end{proof}
\subsection{Bounds on Error in terms of Residuals}
\label{sec:32}
In the previous subsection (Corollary \ref{cor:er}), we showed that the residuals appearing in the loss function \eqref{eq:minimax} can be made arbitrarily small. Hence, this loss function constitutes a suitable target for the training process. However, as argued in \cite{deryck2021approximation}, the smallness of residual, does not necessarily imply that the total error, with respect to PINNs, can be made arbitrarily small as the optimization process might converge to local saddle-points of \eqref{eq:minimax}. Hence, we need to bound the error in terms of the corresponding residuals. We obtain such bounds in this section. 

To start with, we define the following set of test functions, 
\begin{definition}\label{def:varphi}
Let for any $(y,s)\in [0,1]\times[0,T]$ and $\epsilon>0$ the function $\overline{\varphi}_\epsilon^{y,s}:[0,1]\times[0,T]\to [0,\infty)$ be given by, 
\begin{equation}\label{eq:varphi}
\begin{split}
    \overline{\varphi}^{y,s}_\epsilon(x,t) &= \chi_\epsilon\left(\frac{t+s}{2}\right)\rho_\epsilon(x-y)\rho_\epsilon(t-s),\\
    \chi_\epsilon(t) &= \frac{1}{2\sigma(\alpha\epsilon)}(\sigma(\alpha( t-2\epsilon))-\sigma(\alpha (t-T+2\epsilon))),  \qquad \alpha = 3\ln(1/\epsilon)/\epsilon,\\
    \rho_\epsilon(x) &= \frac{\sigma(\beta(x+{\epsilon^6}))-\sigma(\beta(x-{\epsilon^6}))}{2\epsilon^6}, \qquad \beta = 9\ln(1/\epsilon)/\epsilon^3, 
\end{split}
\end{equation}
for $(x,t)\in [0,1]\times[0,T]$. Furthermore we define the set $\Phi_\epsilon$ by, 
\begin{equation}
    \Phi_\epsilon = \left\{\overline{\varphi}^{y,s}_\epsilon\::\: (y,s)\in [0,1]\times[0,T]\right\}. 
\end{equation}
\end{definition} 
Now, we will modify the famous doubling of variables argument of Kruzkhov to obtain the following bound on the $L^1$-error with \emph{wPINNs},
\begin{theorem}\label{thm:stability}
Assume that $u$ is the piecewise smooth entropy solution of \eqref{eq:scl} with essential range $\cC$ and that $u(0,t)=u(1,t)$ for all $t\in[0,T]$. 
There is a constant $C>0$ such that for every $\epsilon>0$ and $v\in C^1(D\times[0,T])$, it holds that 
\begin{equation}\label{eq:stability}
    \begin{split}
        \int_0^1 \abs{v(x,T)-u(x,T)}dx &\leq C\bigg(\int_0^1 \abs{v(x,0)-u(x,0)}dx + \max_{c\in\cC, \varphi\in\Phi_\epsilon}\cR(v,\varphi, c) \\
        &\quad + (1+\norm{v}_{C^1})\ln(1/\epsilon)^3\epsilon + \int_0^T \abs{v(1,t)-v(0,t)}dt\bigg).
    \end{split}
\end{equation}
\end{theorem}

\begin{proof}
Since $u$ is an entropy solution, inequality \eqref{eq:kruzkov} in Definition \ref{def:entropy} is satisfied in particular for $c\leftarrow v(y,s)$ for any $(y,s)\in D\times [0,T]$. We now apply Lemma \ref{lem:partial-int-t} with $z\leftarrow \abs{u(x,t)-v(y,s)}$ to obtain that,
\begin{equation}\label{eq:A-step1a}
\begin{split}
&- \int_0^1\int_0^T \int_0^1\int_0^T \abs{u(x,t)-v(y,s)} \partial_t \overline{\varphi}_\epsilon^{y,s}(x,t) dtdxdsdy \\
&\quad \leq \int_0^1\int_0^T \int_0^1\int_0^T \partial_t \abs{u(x,t)-v(y,s)}  \overline{\varphi}_\epsilon^{y,s}(x,t) dtdxdsdy +CB\epsilon,
\end{split}
\end{equation}
and Lemma \ref{lem:partial-int-x} to obtain that,
\begin{equation}\label{eq:A-step1b}
\begin{split}
&- \int_0^1\int_0^T \int_0^1\int_0^T Q[u(x,t);v(y,s)] \partial_x \overline{\varphi}_\epsilon^{y,s}(x,t) dtdxdsdy \\
&\quad \leq \int_0^1\int_0^T \int_0^1\int_0^T \partial_x Q[u(x,t);v(y,s)]  \overline{\varphi}_\epsilon^{y,s}(x,t) dtdxdsdy\\
&\qquad+12L_f \int_0^T \abs{v(1,t)-v(0,t)}dt+CL_f\norm{v_x}_{\infty} (1-\ln(\epsilon))\epsilon, 
\end{split}
\end{equation}

Next, we observe that for $v\in C^1(D\times [0,T])$ and $(y,s),(x,t)\in D\times [0,T]$ it holds that,
\begin{equation}\label{eq:A-step1}
    \cR(v,\overline{\varphi}_\epsilon^{x,t}, u(x,t))\leq \max_{c, \varphi}\cR(v,\varphi, c),
\end{equation}
where $\cC$ is the essential range of $u$. As a result, we find using the symmetry of $Q$ that, 
\begin{equation}\label{eq:A-step2}
    \begin{split}
        &- \int_0^1\int_0^T \int_0^1\int_0^T \left(\abs{u(x,t)-v(y,s)} \partial_s \overline{\varphi}_\epsilon^{x,t}(y,s)+Q[u(x,t);v(y,s)] \partial_y \overline{\varphi}_\epsilon^{x,t}(y,s)\right) dsdydtdx\\
        &\qquad\leq T \max_{c, \varphi}\cR(v,\varphi, c).
    \end{split}
\end{equation}
Summing \eqref{eq:A-step1a}, \eqref{eq:A-step1b} and \eqref{eq:A-step2} and using Fubini's theorem and the fact that $\overline{\varphi}_\epsilon^{x,t}(y,s) = \overline{\varphi}_\epsilon^{y,s}(x,t)$ leads us to,
\begin{equation}\label{eq:A-step3}
\begin{split}
&- \int_D\int_{[0,T]}\int_D\int_{[0,T]} \abs{v(x,t)-u(y,s)}(\partial_t \overline{\varphi}_\epsilon^{y,s}(x,t)+\partial_s \overline{\varphi}_\epsilon^{y,s}(x,t))  dx dt dy ds \quad {\scriptstyle =:\:T_1}\\
    &\quad- \int_D\int_{[0,T]}\int_D\int_{[0,T]} Q[v(x,t); u(y,s)](\partial_x \overline{\varphi}_\epsilon^{y,s}(x,t)+\partial_y \overline{\varphi}_\epsilon^{y,s}(x,t)) dx dt dy ds \quad {\scriptstyle =:\:T_2}\\
    &\leq T\max_{c\in\cC, \varphi\in\Phi}\cR(v,\varphi, c) + C(1+\norm{v_x}_\infty)(1-\ln(\epsilon))\epsilon + C\int_0^T \abs{v(1,t)-v(0,t)}dt.
\end{split}
\end{equation}
One can observe that $T_2=0$ since $\partial_x \overline{\varphi}_\epsilon^{y,s} = -\partial_y \overline{\varphi}_\epsilon^{y,s}$. In what follows, we will prove that $T_1$ is a good approximation of $ \norm{v(T)-u(T)}_{L^1}- \norm{v(0)-u(0)}_{L^1}$. 
By replacing the absolute value $\abs{\cdot}$ in $T_1$ by its smoothened counterpart $\abs{\cdot}_\eta$, $\eta\geq0$, (as defined in Lemma \ref{lem:smooth-abs}) and by using the definition of $\overline{\varphi}_\epsilon^{y,s}(x,t)$ we find that $T_1=T_1^0$ where,
\begin{equation}
    T_1^\eta := - \int_D\int_{[0,T]}\int_D\int_{[0,T]} \abs{v(y,s)-u(x,t)}_\eta \chi'_\epsilon\left(\frac{t+s}{2}\right)\rho_\epsilon(x-y)\rho_\epsilon(t-s)  dx dt dy ds. 
\end{equation}
Using Lemma \ref{lem:smooth-abs}, we find for $\eta>0$ that $\abs{T_1-T_1^\eta}\leq C\eta$
for some absolute constant $C>0$.
For every $x,y,t$, we now want to apply Lemma \ref{lem:rho} with $f\leftarrow(s\mapsto \abs{v(y,s)-u(x,t)}_\eta \chi'_\epsilon\left(\frac{t+s}{2}\right))$. Define $\omega(t) = \sigma(\beta(t-\max\{0,t-\epsilon^3\}))-\sigma(\beta(t-\min\{b,t+\epsilon^3\}))$. We find using Lemma \ref{lem:smooth-abs} that
\begin{equation}
\begin{split}
    \abs{T_1^\eta-\Tilde{T}_1^\eta}&:=\abs{T_1^\eta + \int_D\int_{[0,T]}\int_D \abs{v(y,t)-u(x,t)}_\eta \chi'_\epsilon\left(t\right)\rho_\epsilon(x-y)\omega(t) dx dy dt}\\
    & \leq 10(4B\alpha^2+4\norm{v_t}_\infty)(T-3\ln(\epsilon))\epsilon^3.
\end{split}
\end{equation}
Next, for every $x,t$, we now want to apply Lemma \ref{lem:rho} with $f\leftarrow(y\mapsto \abs{v(y,t)-u(x,t)}_\eta )$. We find using Lemma \ref{lem:smooth-abs} that, 
\begin{equation}
    \begin{split}
    \abs{\Tilde{T}_1^\eta-\hat{T}_1^\eta(D)}&:=\abs{\Tilde{T}_1^\eta + \int_D\int_{[0,T]} \abs{v(x,t)-u(x,t)}_\eta \omega(x)\omega(t) \chi'_\epsilon\left(t\right) dx dt}\\
    & \leq 10\cdot 4\norm{v_t}_\infty(1-3\ln(\epsilon))\epsilon^3. 
\end{split}
\end{equation}
Let $A = \{x\in D \vert \text{ the function } [0,3\epsilon]\cup[T-3\epsilon,T]\to\R: t\mapsto u(x,t) \text{ is cont. diff.}\}$. As $u$ is piecewise smooth, we find for some constant $C>0$ that $\abs{\hat{T}_1^\eta(D)-\hat{T}_1^\eta(A)}\leq C\epsilon$.
We can now use Lemma \ref{lem:chi} for every $x\in A$ with $f\leftarrow (t\mapsto \abs{v(x,t)-u(x,t)}_\eta \omega(t))$. Let $B = [\epsilon, 3\epsilon]\cup[T-3\epsilon, T-\epsilon]$. Note that $\omega(t)$ is constant on $B$ under the assumption that $\epsilon<1$ and therefore $\epsilon^3<\epsilon$. This gives us, writing $T_\epsilon := T-2\epsilon$, 
\begin{equation}
\begin{split}
&\abs{\hat{T}_1^\eta(A)-\cT_1^\eta(A)}:=\\
&=\abs{\hat{T}_1^\eta(A)-\int_A\abs{v(x,T_\epsilon)-u(y,T_\epsilon)}_\eta \omega(x)\omega(T_\epsilon)dx+\int_A\abs{v(x,2\epsilon)-u(y,2\epsilon)}_\eta \omega(x)\omega(2\epsilon)dx}\\
&\leq C(T\abs{\cC}+(\norm{v_t}_\infty+\abs{u}_{W^{1,\infty}(B)})\ln(1/\epsilon))\frac{\epsilon}{1-\epsilon}.
\end{split}
\end{equation}
We find also that $\abs{\cT_1^\eta(A)-\cT_1^\eta(D)}\leq C\epsilon$.
Next, using the time continuity of entropy solutions (Lemma \ref{thm:entropy-sol}),
\begin{equation}
\begin{split}
    \norm{v(T)-u(T)}_{L^1} &\leq \norm{v(T)-v(T_\epsilon)}_{L^1}+\norm{v(T_\epsilon)-u(T_\epsilon)}_{L^1}+\norm{u(T_\epsilon)-u(T)}_{L^1}\\
    &\leq 2\abs{D}\norm{v_t}_\infty \epsilon + \norm{v(T_\epsilon)-u(T_\epsilon)}_{L^1} +2\epsilon M \norm{u_0}_{TV},
\end{split}
\end{equation}
where $M = \max_{c\in \cC}\abs{f'(c)}$.
Similarly, 
\begin{equation}
    \begin{split}
        \norm{v(2\epsilon)-u(2\epsilon)}_{L^1} &\leq \norm{v(2\epsilon)-v(0)}_{L^1}+\norm{v(0)-u(0)}_{L^1}+\norm{u(0)-u(2\epsilon)}_{L^1}\\
        &\leq 2\abs{D}\norm{v_t}_\infty \epsilon+\norm{v(0)-u(0)}_{L^1}+2\epsilon M \norm{u_0}_{TV}.
    \end{split}
\end{equation}
We can now bring everything together. We will only quantify the dependence on $v$ and aggregate all other constants in the constant $C>0$, which we update progressively. If we set $\eta\leftarrow \epsilon$ we find that, 
\begin{equation}
    \begin{split}
        \norm{v(T)-u(T)}_{L^1} 
        &\leq C\left(\norm{v(0)-u(0)}_{L^1} +T_1+\eta + (\alpha^2+\norm{v_t}_\infty)\ln(1/\epsilon^3)\frac{\epsilon^3}{1-\epsilon}\right)\\
        & \leq C \left(\norm{v(0)-u(0)}_{L^1} +T_1+(1+\norm{v_t}_\infty)\ln(1/\epsilon)^3\epsilon\right). 
    \end{split}
\end{equation}
Combining this with \eqref{eq:A-step2} and the fact that $T_2=0$ brings us, 
\begin{equation}
    \begin{split}
        \int_0^1 \abs{v(x,T)-u(x,T)}dx &\leq C\bigg(\int_0^1 \abs{v(x,0)-u(x,0)}dx + \max_{c\in\cC, \varphi\in\Phi_\epsilon}\cR(v,\varphi, c) \\
        &\quad + (1+\norm{v}_{C^1})\ln(1/\epsilon)^3\epsilon + \int_0^T \abs{v(1,t)-v(0,t)}dt\bigg). 
    \end{split}
\end{equation}

\end{proof}

\begin{remark}
\label{rem:tder}
In experiments (cf. Section \ref{sec:weak-pinn-training}), one can replace $\cR$ with the following alternative,
\begin{equation}\label{eq:R-equivalent}
    \Tilde{\cR}(v,\varphi, c) := \int_D\int_{[0,T]} \left(\varphi(x,t)\partial_t\abs{v(x,t)-c} - Q[v(x,t); c]\partial_x\varphi(x,t) \right)dx dt. 
\end{equation}
The only difference with $\cR$ is that in $\Tilde{\cR}$ the time derivative is with $\abs{v(x,t)-c}$ and not with the test function. 
Because of Lemma \ref{lem:partial-int-t} it holds that $\abs{\Tilde{\cR}(v,\varphi, c)-{\cR}(v,\varphi, c)} = \bigO(\epsilon)$ if $\varphi\in \Phi_\epsilon$. 
\end{remark}
\subsection{Bounds on the Generalization Gap}
\label{sec:33}
The main point of Theorem \ref{thm:stability} was to provide an upper bound on the $L^1$-error in terms of the residuals. However, in practice, one cannot evaluate the integrals in the residuals, for instance in \eqref{eq:R-def}, exactly and has to resort to quadrature. To this end, we consider the simplest case of random (Monte Carlo) quadrature and generate a set of collocation points, $\S = \{(x_i, t_i)\}_{i=1}^{M} \subset D\times [0,T]$, where all $(x_i, t_i)$ are iid drawn from the uniform distribution on $D\times [0,T]$. 
For a fixed $\theta\in\Theta$, $\varphi\in\Phi_\epsilon$, $c\in\cC$ and for this data set $\S$ we can then define the \emph{training error}, 
\begin{equation}\label{eq:def-training-error}
    \begin{split}
        \Et(\theta, \S,\varphi, c) &= -\frac{T}{M} \sum_{i=1}^{M} \left(\abs{u_\theta(x_i,t_i)-c}\partial_t\varphi(x_i,t_i) + Q[u_\theta(x_i,t_i); c]\partial_x\varphi(x_i,t_i) \right)\\
        &\qquad + \frac{T}{M}\sum_{i=1}^{M} \abs{u_\theta(x_i,0)-u(x_i,0)} +\frac{T}{M}\sum_{i=1}^{M}\abs{u_\theta(0,t_i)-u_\theta(1,t_i)}.
    \end{split}
\end{equation}

During training, one then aims to obtain neural network parameters $\theta^*_\S$, a test function $\varphi^*_\S$ and a scalar $c^*_\S$ such that
\begin{equation}
     \Et(\theta^*_\S, \S,\varphi^*_\S, c^*_\S)\approx \min_{\theta\in\Theta} \max_{\varphi\in\Phi_\epsilon}\max_{c\in\cC} \Et(\theta, \S,\varphi, c),
\end{equation}
for some $\epsilon>0$. We call the resulting neural network $u^*:=u_{\theta^*_\S}$ a weak PINN (\emph{wPINN}).

Note that the training error \eqref{eq:def-training-error} is a \emph{quadrature approximation} of the corresponding (total) residual, 
\begin{equation}\label{eq:def-gen-error}
    \begin{split}
        \Eg(\theta,\varphi, c) &= - \int_0^1\int_0^T \left(\abs{u_{\theta^*(\S)}(x,t)-c}\partial_t\varphi(x,t) + Q[u_{\theta^*(\S)}(x,t); c]\partial_x\varphi(x,t) \right)dx dt\\
        &\qquad + \int_0^1 \abs{u_\theta(x,0)-u(x,0)} dx + \int_0^T \abs{u_\theta(0,t)-u_\theta(1,t)} dt
    \end{split}
\end{equation}
The next step in the analysis of \emph{wPINNs} is to provide a bound on the so-called \emph{generalization gap} i.e., the difference between $\Eg$ and $\Et$. To this end, we have the following theorem,
\begin{theorem}\label{thm:generalization}
Let  $M, L, W\in\mathbb{N}$, $R\geq \max\{1, T, \abs{\cC}\}$ with $L\geq 2$ and $M\geq 3$. Moreover let $u_\theta:D\times [0,T]\to\mathbb{R}$, $\theta\in \Theta,$ be tanh neural networks with at most $L-1$ hidden layers, width at most $W$ and weights and biases bounded by $R$. Assume that $\Eg$ and $\Et$ are bounded by $B\geq 1$. It holds with a probability of at least $1-\delta$ that,
\begin{equation}
    \Eg(\theta^*_\S,\varphi^*_\S, c^*_\S) \leq \Et(\theta^*_\S, \S,\varphi^*_\S, c^*_\S) + \frac{3BdLW}{\sqrt{M}}\sqrt{\ln(\frac{C\ln(1/\epsilon)WRM}{\epsilon^3 \delta B})}
\end{equation}
where $C>0$ is a constant that only depends on $u$ and $f$.
\end{theorem}

\begin{proof}
Note that $\Eg$ \eqref{eq:def-gen-error} is defined as the sum of the error related to the entropy residual $\cR$, the initial condition and the boundary condition. We will bound (with a high probability) each of these terms separately in terms of the corresponding term in the definition of the training error \eqref{eq:def-training-error} using Lemma \ref{lem:bound-generalization}. 

We start by calculating the Lipschitz constant of $\cR$. From the proof of Theorem \ref{thm:approx-general} and from \cite[Lemma 11]{deryck2021pinn} we find that,
\begin{equation}
\begin{split}
    \abs{\cR(u_\theta,\varphi, c)-\cR(u_\vartheta,\varphi, c)} &\leq T \abs{\varphi}_{W^{1,\infty}} (1+3L_f)\max_{x,t}\abs{(u_\theta-u_\vartheta)(x, t)}\\
    &\leq T \abs{\varphi}_{W^{1,\infty}} (1+3L_f) (d+5)(WR)^{L-1}\norm{\theta-\vartheta}_\infty.
\end{split}
\end{equation}
Following the steps of the proof of Lemma \ref{lem:Q} we also find that, 
\begin{equation}
    \abs{\cR(v,\varphi, c)-\cR(v,\varphi, b)} \leq (1+3L_f)\abs{\varphi}_{W^{1,1}(D\times [0,T])} \abs{b-c}. 
\end{equation}
Using that $\partial_t \overline{\varphi}_\epsilon = \alpha \overline{\varphi}_\epsilon +\beta \overline{\varphi}_\epsilon$ and $\partial_t \overline{\varphi}_\epsilon = \beta \overline{\varphi}_\epsilon$ we find,
\begin{equation}
    \begin{split}
        &\abs{\cR(v,\overline{\varphi}_\epsilon^{y,s}, c)-\cR(v,\overline{\varphi}_\epsilon^{z,\tau}, c)}\\
        &\leq \int_0^1\int_0^T \abs{\abs{v(x,t)-c}\partial_t(\overline{\varphi}_\epsilon^{y,s}-\overline{\varphi}_\epsilon^{z,\tau})(x,t) + Q[v(x,t); c]\partial_x(\overline{\varphi}_\epsilon^{y,s}-\overline{\varphi}_\epsilon^{z,\tau})(x,t) }dt dx\\
        &\leq C(\alpha+\beta)\int_0^1\int_0^T\abs{(\overline{\varphi}_\epsilon^{y,s}-\overline{\varphi}_\epsilon^{z,\tau})(x,t)} dt dx, 
    \end{split}
\end{equation}
where $C>0$ is a constant depending on $u$ and $f$. Combining the previous calculations with $\alpha<\beta$
and
$\abs{\varphi}_{W^{1,\infty}} =  \bigO(\beta^{3})$ (Lemma \ref{lem:Phi}) we find that, 
\begin{equation}
    \abs{\cR(v,\overline{\varphi}_\epsilon^{y,s}, c)-\cR(v,\overline{\varphi}_\epsilon^{z,\tau}, c)}\leq C\beta^4(\abs{y-z}+\abs{s-\tau}). 
\end{equation}
Putting everything together, we find that, 
\begin{equation}
\begin{split}
     \abs{\cR(u_\theta,\overline{\varphi}_\epsilon^{y,s}, c)-\cR(u_\vartheta,\overline{\varphi}_\epsilon^{z,\tau}, b)}&\leq C\beta^4(d+5)(WR)^{L-1}\norm{(\theta,y,s,c)-(\vartheta,z,\tau,b)}_\infty. 
\end{split}
\end{equation}
Similarly, we also find that
\begin{equation}
\begin{split}
    &\abs{\int_0^1 \abs{u_\theta(x,0)-u(x,0)} dx-\int_0^1 \abs{u_\vartheta(x,0)-u(x,0)} dx}  \leq \max_{x}\abs{(u_\theta-u_\vartheta)(x, 0)}\leq (d+5)(WR)^{L-1}\norm{\theta-\vartheta}_\infty
\end{split}
\end{equation}
and also,
\begin{equation}
    \abs{\int_0^T \abs{u_\theta(0,t)-u_\theta(1,t)} dt-\int_0^T \abs{u_\vartheta(0,t)-u_\vartheta(1,t)} dt} \leq T(d+5)(WR)^{L-1}\norm{\theta-\vartheta}_\infty. 
\end{equation}
In total, we conclude that
\begin{equation}
\begin{split}
    \abs{\Eg(\theta,y,s,c)-\Eg(\vartheta,z,\tau,b)}  &\leq C\beta^4(d+5)(WR)^{L-1}\norm{(\theta,y,s,c)-(\vartheta,z,\tau,b)}_\infty  \\ &=: \mathfrak{L}\norm{(\theta,y,s,c)-(\vartheta,z,\tau,b)}_\infty. 
\end{split}
\end{equation}
The same inequality holds for the training error \eqref{eq:def-training-error}. 

Next, one can calculate that every $u_\theta$ has at most $(d+1+(L-2)W+1)W$ weights and $(L-1)W+1$ biases. Together with $y,s,b$ we find that there are at most $k\leftarrow 4(d+1)LW^2\leq 9dLW^2$ parameters. We can now obtain the inequality from the statement by using Lemma \ref{lem:bound-generalization} with $\Theta\leftarrow\Theta\times [0,1]\times[0,T]\times\cC$, the calculated values of $\mathfrak{L}, k$ and $a\leftarrow R$ and the estimate,
\begin{equation}
     \sqrt{\frac{B^2k}{M}\ln(\frac{a\mathfrak{L}\sqrt{M}}{\sqrt[k]{\delta} B})} \leq \frac{3BdLW}{\sqrt{M}}\sqrt{\ln(\frac{C\ln(1/\epsilon)WRM}{\epsilon^3 \delta B})}, 
\end{equation}
where we used that $\beta = \bigO(\ln(1/\epsilon)\epsilon^{-3})$. Finally, our chosen value of $k$ leads to the implication that $M\geq 3 \implies M\geq e^{16/k}$, as is required by Lemma \ref{lem:bound-generalization}.  
\end{proof}

Using Theorem \ref{thm:stability} and the above bound on the generalization gap, one can prove the following rigorous upper bound on the total $L^1$-error of the weak PINN, which we denote as,
\begin{equation}\label{eq:def-total-error}
\begin{split}
        \cE^T(\theta) &= \int_D \abs{u_\theta(x,T)-u(x,T)} dx. 
    \end{split}
\end{equation}

\begin{corollary}\label{cor:generalization}
Assume the setting of Theorem \ref{thm:generalization}. It holds with a probability of at least $1-\delta$ that
\begin{equation}
\label{eq:erest}
\begin{split}
     \cE^T(\theta^*_\S) &\leq C \Bigg[\underbrace{\Et(\theta^*_\S, \S,\varphi^*_\S, c^*_\S), c^*_\S)}_{\text{training error}}+ \underbrace{\frac{3BdLW}{\sqrt{M}}\sqrt{\ln(\frac{C\ln(1/\epsilon)WRM}{\epsilon^3 \delta B})}}_{\text{generalization gap}} 
     + \underbrace{(1+\norm{u^*}_{C^1})\ln(1/\epsilon)^3\epsilon}_{\text{limited test space } \Phi_\epsilon \subsetneq H^1} 
     \Bigg]. 
\end{split}
\end{equation}
\end{corollary}
\begin{proof}
Using Theorem \ref{thm:stability} we find that,
\begin{equation}
\begin{split}
    \cE^T(\theta^*_\S) &\leq 
    C\bigg[\norm{u^*(\cdot,0)-u(\cdot,0)}_{L^1} + \max_{c\in\cC, \varphi\in\Phi_\epsilon}\cR(u^*,\varphi, c) \\
        &\qquad + (1+\norm{u^*}_{C^1})\ln(1/\epsilon)^3\epsilon + \norm{u^*(1,\cdot)-u(0,\cdot)}_{L^1}\bigg] \\
&\leq 
    C\bigg[\cR(u^*,\varphi^*_\S, c^*_\S)+\norm{u^*(\cdot,0)-u(\cdot,0)}_{L^1} +(1+\norm{u^*}_{C^1})\ln(1/\epsilon)^3\epsilon \\
        &\qquad + \left(\max_{c\in\cC, \varphi\in\Phi_\epsilon}\cR(u^*,\varphi, c) -\cR(u^*,\varphi^*_\S, c^*_\S)\right) + \norm{u^*(1,\cdot)-u(0,\cdot)}_{L^1}\bigg]\\
        &=  C\bigg[ \Eg(\theta^*_\S,\varphi^*_\S, c^*_\S) + (1+\norm{u^*}_{C^1})\ln(1/\epsilon)^3\epsilon +\underbrace{\max_{c\in\cC, \varphi\in\Phi_\epsilon}\cR(u^*,\varphi, c) -\cR(u^*,\varphi^*_\S, c^*_\S)}_{\leq 0} \bigg]
\end{split}
\end{equation}
Combining the previous inequality with Theorem \ref{thm:generalization} we then find the statement of the corollary. 
\end{proof}
Thus, the estimate \eqref{eq:erest} provides a rigorous bound on the total error of a \emph{wPINN} in approximating the entropy solution of the scalar conservation law \eqref{eq:scl}, in terms of the training error, size of the underlying neural networks and the number of collocation points. In particular, this error can be made as small as desired by choosing sufficient number of collocation points and networks of suitable size. 

\begin{remark}
Instead of using random training points, one can choose a training set based on low-discrepancy sequences. This will improve the convergence rate in terms of $M$ to $\bigO(V_{HK}\ln(M)^d M^{-1})$ by using arguments from \cite{mishra2021enhancing}, where $V_{HK}$ is the Hardy-Krause variation of the integrand in the definition of $\Eg(\theta,c)$. However, if one also considers the $\epsilon$-dependence of $V_{HK}$ through $\overline{\varphi}_{\epsilon}$ then it is better to use random points as the bound in Corollary \ref{cor:generalization} only depends sublogarithmically on $\epsilon^{-1}$ whereas upper bounds for $V_{HK}$ depend polynomially on $\epsilon^{-1}$. 
\end{remark}

\begin{remark}\label{rmk:lnorm}
Throughout the entire section, we have focused on calculating the $L^1$-error of the weak PINN, as the $L^1$-norm is a natural choice for scalar conservation laws. In practice, however, we will see that it is easier to train the weak PINN using an $L^2$-based loss function. The developed theory can be easily extended to $L^p$-based loss functions for $p>1$. First, one can use Hölder's inequality to prove an $L^2$-based upper bound for the RHS of the stability result \eqref{eq:stability} of Theorem \ref{thm:stability}. Second, one needs to make an adaptation to Theorem \ref{thm:generalization}. This will lead to a theoretical deterioration of the generalization gap from $\bigO\left(M^{-1/2}\sqrt{\ln(M)}\right)$ to $\bigO\left(M^{-1/4}\sqrt{\ln(M)}\right)$, as is common for these type of estimates. 
\end{remark}
\section{Training and Implementation of \emph{wPINNs}}
\label{sec:weak-pinn-training}
In this section, we describe how \emph{wPINNs}, designed in Section \ref{sec:24}, are implemented and trained in practice. To this end, we need the following elements,
\subsection{Set of Collocation Points.}
The set of collocation points $\S$, also called as \emph{training set} in the literature on PINNs, was already introduced in Section \ref{sec:33}. 
We will divide $\S \subset D \times  [0,T]$, into the following three parts,
\begin{itemize}
    \item Interior collocation points $\S_{int}=\{y_m\}$ for $1 \leq m \leq M_{int}$, with each $y_m = (x,t)_m \in D\times [0, T]$.  
    \item Spatial boundary collocation points $\S_{sb} = \{z_m\}$ for $1 \leq m \leq M_{sb}$ with each $z_m= (x,t)_m$ and $z_m \in \partial D \times [0,T]$. 
    \item Temporal boundary collocation points $\S_{tb} = \{x_m\}$, with $1 \leq m \leq M_{tb}$ and $x_m \in D$.
\end{itemize}
The full set of collocation points is $\S = \S_{int} \cup \S_{sb} \cup \S_{tb}$ and $M = M_{int} + M_{sb} + M_{tb}$.
\subsection{Residuals}
We recapitulate the definitions of different residuals, already introduced in \eqref{eq:minimax} and further refined in Theorem \ref{thm:stability} (see also Remark \ref{rem:tder}) below,
\begin{itemize}
    \item \textit{Interior residual} given by,
    \begin{equation}
        r_{int}[u_\theta, \varphi](x,t, c):= \varphi(x,t)\partial_t\abs{u_\theta(x,t)-c} - Q[u_\theta(x,t); c]\partial_x\varphi(x,t) .
    \end{equation}
Observe that the residual depends on the test function $\varphi(x,t)$. We restrict the choice of the test function to the parametrized family of functions $\varphi_\eta(x,t)$ defined as $\varphi_\eta(x,t) = \omega(x,t) \xi_\eta(x,t)$.  Here,  $\omega : D \mapsto \R$ is a \textit{cutoff} function satisfying the following properties:
\begin{enumerate}
    \item $\omega(x)=1,\quad x\in D_\varepsilon$,
    \item $\omega(x)=0,\quad x\in \partial D$,
\end{enumerate}
\begin{equation}
    D_\epsilon = \{x \in D: dist(x,\partial D)< \epsilon\},
\end{equation}
and $\xi_\eta(x,t)$ a neural network with trainable parameters $\eta$. 
The choice of the cutoff function guarantees that the test function has compact support.
Then, the interior residual becomes,
    \begin{equation}
         r_{int}[u_\theta, \varphi_\eta]:= \varphi_\eta(x,t)\partial_t\abs{u_\theta(x,t)-c} - Q[u_\theta(x,t); c]\partial_x\varphi_\eta(x,t).
    \end{equation}
\item \textit{Spatial boundary residual} given by,
\begin{equation}
    \label{eq:hres2}
    r_{sb}[u_\theta](x,t):= u_{\theta}(x,t) - g(x,t).\quad \forall x \in \partial D, ~ t \in (0,T],
\end{equation}
Although the estimates above were derived assuming periodic boundary conditions, the numerical experiments are carried out with Dirichlet boundary conditions corresponding to the boundary trace of exact solutions i.e., with $g(x,t)=u|_{\partial D \times (0,T)}$.
\item \textit{Temporal boundary residual} given by,
\begin{equation}
    \label{eq:hres3}
    r_{tb}[u_\theta](x):= u_{\theta}(x,0) - {u}(x, 0), \quad \forall x \in D.
\end{equation}
\end{itemize}
\subsection{Loss function}
Given the definitions above, we consider the following loss function,
\begin{equation}
    \label{eq:hlf}
    J(\theta, \eta) = J_{pde}(\theta, \eta,c) +  \lambda J_{u}(\theta)
\end{equation}
with 
\begin{equation}
   J_{pde}(\theta, \eta, c) = \frac{\left(ReLU\Big(\sum_{m=1}^{M_{int}} r_{int}[u_\theta, \varphi_\eta](y_m, c)\Big)\right)^2}{\sum_{m=1}^{M_{int}} \partial_x \varphi_\eta(y_m)^2}, \quad J_u(\theta) = \sum\limits_{m=1}^{M_{tb}}  r_{tb}[u_\theta](x_m)^2 + \sum\limits_{m=1}^{M_{sb}}  r_{sb}[u_\theta](z_m)^2.
\end{equation}
Here, $\lambda$ is a hyperparameter for balancing the role of PDE and data residuals, and the denominator of $J_{pde}$ a Monte Carlo approximation of the $H^1$-seminorm of $\varphi_\eta(x,t)$. The following remarks about the specific form of the loss function \eqref{eq:hlf} are in order, 
\begin{remark}
The terms appearing in the loss function terms are Monte Carlo approximations of the integrals in the error estimate, cf. Theorem \ref{thm:stability}, where the $L^1$ norm has been replaced by the $L^2$ (Remark \ref{rmk:lnorm}). This was done to facilitate optimization of the resulting min-max problem. 
\end{remark}
\begin{remark}
We observe that additional terms, i.e., use of the ReLU function and test function $H^1$-seminorm, have been introduced in the definition of the loss function \eqref{eq:hlf}, when compared to the terms in the error estimate in Theorem \ref{thm:stability}. These are introduced to facilitate training and the estimates can be readily extended to incorporate the contributions of these terms.  
\end{remark}
\begin{remark}
In practice, the maximization problem with respect to the scalar $c$ is solved by computing $J_{max,C}(\theta, \eta) = \max_{c_i\in C}J_{max}(\theta, \eta, c_i)$, for a discrete set of values $C= \{ c_i \}_{i=1}^M, c_i\in [c_{min}, c_{max}]$, whereas the optimization problems with respect to the neural network parameters $\theta$ and $\eta$ is approximated with gradient descent and ascent, respectively.
\end{remark}
\subsection{Weak PINNs Algorithm}
Given the above description of its constituents, we summarize the \emph{wPINNs} algorithm in Algorithm \ref{alg:wpinn}. \\\\
\begin{algorithm}[H]
\SetAlgoLined
\KwResult{$\theta^\ast_\S, \eta^\ast_\S, c_\S^\ast$}
 Initialize the networks $u_\theta, \varphi_\eta:D \times [0, T]  \mapsto \R $ and $C$\;
 \For{$e=1,...,N$}{
 \For{$k=1,...,N_{max}$}{
 Compute $J_{max,C}(\theta, \eta) = \max_{c_i\in C}J_{max}(\theta, \eta, c_i)$\;
 Update $\eta  \leftarrow \eta + \tau_\eta\nabla J_{max,C}(\theta, \eta)$\;}
 \For{$k=1,...,N_{min}$}{
 Compute $J_{max,C}(\theta, \eta) = \max_{c_i\in C}J_{max}(\theta, \eta, c_i)$\;
 Update $\theta  \leftarrow \theta - \tau_\theta\nabla (\lambda J_{max,C} + J_u)(\theta, \eta)$\;}
 }
 \label{alg:wpinn}
 \caption{Training of \emph{wPINNs}}
\end{algorithm}
\subsection{Implementation of \emph{wPINNs}}
The implementation of the \emph{wPINNs} algorithm is carried out with a collection of Python scripts, realized with the support of PyTorch \url{https://pytorch.org/}. The scripts can be downloaded from \url{https://github.com/mroberto166/wpinns}. Below, we describe key implementation details.
\subsubsection{Ensemble Training}\label{sec:ensemble_train}
 \emph{wPINNs} include several hyperparameters, including number of hidden layers and neurons of the networks, $L_\theta, L_\eta$, $l_\theta$, $l_\eta$, the number of iterations $K_{max}$, $K_{min}$, number of epochs $e$, residual parameter $\lambda$, etc. A user is always confronted with the question of which parameter to choose. It is standard practice in machine learning to perform a systematic hyperparameter search. To this end, we follow the \emph{ensemble training} procedure of \cite{LMR1}: for each configuration of the model hyperparameters we \emph{retrain} the \emph{wPINN} $n_\theta$ times, each with different initialisation of the networks hyperparameters, and select the hyperparameter configuration that minimises the average value over the retrainings of the training error \ref{eq:def-training-error}, computed in the $L^2$ norm:
 \begin{equation}\label{eq:def-selection_criterion}
     \begin{split}
        \Et(\theta^\ast_\S, \eta^\ast_\S, c^\ast_\S) &=  \sum_{m=1}^{M_{int}} \left(\varphi^\ast(y_m)\partial_t\abs{u^\ast(y_m)-c^\ast_\S} - Q[u_\theta(y_m); c^\ast_\S]\partial_x\varphi^\ast(y_m) \right)^2\\
        &\qquad + \sum_{m=1}^{M_{sb}} \abs{u_\theta(x_m,0)-u(x_m,0)}^2 +\sum_{m=1}^{M_{tb}}\abs{u^\ast(z_m)-u(z_m)}^2.
    \end{split}
 \end{equation}
 \subsubsection{Random Reinitialization of the Test function Parameters}
(Approximate) solutions of min-max problems are significantly harder to reach, when compared to standard minimization (or maximization) problems, as they correspond to saddle points of the underlying loss function. One essential ingredient for improving the numerical stability of the algorithm is the random reinitialization of the trainable parameters $\eta$, corresponding to the test function neural network in \eqref{eq:hlf}. This can be performed with frequency $r_f$. This \textit{reset frequency} can be suitably chosen as any other model hyperparameters through ensemble training. On account of this random reinitialization of the test function parameters $\eta$, the algorithm \ref{alg:wpinn} can be readily modified to yield algorithm \ref{alg:wpinn_reset}, that is used in practice.
\subsubsection{Averages of retrainings} \label{sec:retraining_average}
The final \emph{wPINN} approximation to the solution of the scalar conservation law \eqref{eq:scl} at any given input $(x,t)$, denoted as $u_{av}(x,t)$, is defined as the average over retrainings,
\begin{equation}
    u_{av}(x,t) = \frac{1}{n_\theta}\sum_i^{n_\theta}u_i^\ast(x,t),
\end{equation}
where $u_i^\ast(x,t)$ denotes the predictions of the underlying \emph{wPINN} at $(x,t)$, trained via algorithm \ref{alg:wpinn_reset}, with initial parameters $\theta_i, \eta_i$. This averaging is performed to yield more robust predictions as well as to provide an estimate of the underlying uncertainty in predictions, due to the random initializations of the neural network parameters during training. 
\vspace{0.5cm}

\begin{algorithm}[H]
\SetAlgoLined
\KwResult{$\theta^\ast_\S, \eta^\ast_\S, c_\S^\ast$}
 Initialize the networks $u_\theta, \varphi_\eta:D \times [0, T] \mapsto \R $ and $C$\;
 \For{$e=1,...,N$}{
 \If{$e~\% ~(r_fN) = 0$}{Randomly initialize $\eta$\;}  
 \For{$k=1,...,N_{max}$}{
 Compute $J_{max,C}(\theta, \eta) = \max_{c_i\in C}J_{max}(\theta, \eta, c_i)$\;
 Update $\eta  \leftarrow \eta + \tau_\eta\nabla J_{max,C}(\theta, \eta)$\;}
 \For{$k=1,...,N_{min}$}{
 Compute $J_{max,C}(\theta, \eta) = \max_{c_i\in C}J_{max}(\theta, \eta, c_i)$\;
 Update $\theta  \leftarrow \theta - \tau_\theta\nabla (\lambda J_{max,C} + J_u)(\theta, \eta)$\;}
 }
 \label{alg:wpinn_reset}
 \caption{\textit{Weak PINN} training with random reset of the test function parameters}
\end{algorithm}
\section{Numerical Results} 
\label{sec:numexp}
In this section, we present numerical experiments to illustrate the performance of \emph{wPINNs}. To this end, we consider the scalar conservation law \eqref{eq:scl} in the domain $D = [-1,1]$, with the flux function $f(u) = \frac{1}{2}u^2$. Note that this amounts to considering the well-known inviscid Burgers' equation. We evaluate the performance of \emph{wPINNs}, implemented through algorithm \ref{alg:wpinn_reset}, by computing the (relative) total error at a final time $T$,
\begin{equation}
\cE^T_r(\theta^\ast_\S) = \frac{\int_D \abs{u^\ast(x,T)-u(x,T)} dx}{\int_D \abs{u(x,T)} dx},      
\end{equation}
where $u^{\ast}$ is the prediction of the \emph{wPINN} algorithm \ref{alg:wpinn_reset}. We also compute the space-time relative error, 
\begin{equation}
\cE_{r}(\theta^\ast_\S) = \frac{\int_{D\times [0,T]} \abs{u^\ast(x,t)-u(x,t)} dxdt}{\int_{D\times [0,T]} \abs{u(x,t)} dx dt},
\end{equation}
to assess the performance of \emph{wPINNs} over the entire evolution of the entropy solution. We remark that the integrals in the above error expressions can be readily approximated with Monte Carlo quadratures. We consider the following numerical experiments, 
\subsection{Standing and Moving Shock}
As a first numerical example, we consider the Burgers' equation in $[-1,1]\times [0, 0.5]$  with initial conditions:
\begin{equation}
    u_0(x) = \begin{cases} 
      1 & x\leq 0 \\
      -1 & x> 0 
   \end{cases}, \quad u_0(x) = \begin{cases} 
      1 & x\leq 0 \\
      0 & x> 0 
   \end{cases}
\end{equation}
which result into a standing shock located at $x=0$ and a shock moving with speed $0.5$, respectively:
\begin{equation}\label{eq:ex_st_mov}
    u(x, t) = \begin{cases} 
      1 & x\leq 0 \\
      -1 & x> 0 
   \end{cases}, 
   \quad 
   u(x, t) = \begin{cases} 
      1 & x\leq \frac{t}{2} \\
      0 & x> \frac{t}{2}
   \end{cases}
\end{equation}
We perform an ensemble training, as outlined in the previous section, to find the best set of hyperparameters among those mentioned in Tables \ref{tab:moving_hyp} and \ref{tab:sine_hyp}. On the other hand, we fix $l_\theta=20$, $l_\eta=10$, $N_{min}=1$, $\lambda=10$ $e=5000$, $\tau_\theta = 0.01$ and $\tau_\eta = 0.015$, $n_\theta=10$ and the sin activation function as the activation function $\sigma_\theta$ for the neural network approximating the solution of \eqref{eq:scl}. 
\begin{table}[htbp] 
    \centering
    \renewcommand{\arraystretch}{1.1} 
    
    \footnotesize{
        \begin{tabular}{ c c c c c  c c    } 
            \toprule
              &  \bfseries $L_{\theta}$ & \bfseries $L_{\eta}$ &  \bfseries $\sigma_{\eta}$ &\bfseries $N_{max}$  & \bfseries $r_f$   \\ 
            \midrule
            \midrule
           & 4,6 & 2,4  & $sin$,$tanh$  &6, 8 &0.001, 0.005, 0.025, 0.05\\ 
            \bottomrule
        \end{tabular}
    \caption{Hyperparameter configurations and number of retrainings employed in the ensemble training of \emph{wPINN} for moving shock.}
        \label{tab:moving_hyp}
    }
\end{table}

\begin{table}[htbp] 
    \centering
    \renewcommand{\arraystretch}{1.1} 
    
    \footnotesize{
        \begin{tabular}{ c c c c c c c  c  } 
            \toprule
              &  \bfseries $L_{\theta}$ & \bfseries $L_{\eta}$ & \bfseries $\sigma_{\eta}$ &\bfseries $N_{max}$  & \bfseries $r_f$   \\ 
            \midrule
            \midrule
           & 4,6 & 2,4   & $sin$,$tanh$  &6, 8 &0.025, 0.05, 0.25\\ 
            \bottomrule
        \end{tabular}
    \caption{Hyperparameter configurations and number of retrainings employed in the ensemble training of \emph{wPINN} for standing shock, rarefaction wave and initial sine wave.}
        \label{tab:sine_hyp}
    }
\end{table}

With this setting, the \emph{wPINNs} algorithm \ref{alg:wpinn_reset} is run and its average prediction (as described in Section \ref{sec:retraining_average}) is plotted in Figures \ref{fig:standing} and \ref{fig:moving}, respectively, where we also compare the predictions with the exact solutions \eqref{eq:ex_st_mov} at different times. From these figures, we observe that the \emph{wPINNs} average prediction accurately approximates both the standing as well as the moving shock. There is some variance in the predictions of multiple retrainings. This is completely expected as a highly non-convex min-max optimization problem is being approximated and it is possible to be trapped at local saddle points. Nevertheless, the quantitative predictions of the relative total errors $\cE_r(\theta^\ast_\S)$ and $\cE^T_{r}(\theta^\ast_\S)$, presented in Table \ref{tab:res}, are very accurate, with even errors for the whole time-history of evolution, being below $2\%$. Finally, in Figure  \ref{fig:burg_best_overall}, we also plot the \emph{wPINN} prediction that corresponds to the hyperparameter configuration that leads to smallest overall error among all tested hyperparameter configurations.  We term it as the \emph{best hyperparameter configuration}. This \emph{best} prediction is extremely accurate, with the largest error below $0.2\%$. However, in practice, one does not have access to exact (or reference) solutions and needs to choose hyperparameter configurations that correspond to the smallest values of the loss function. 
\begin{figure}
    \begin{subfigure}{.48\textwidth}
        \centering
        \includegraphics[width=1\linewidth]{{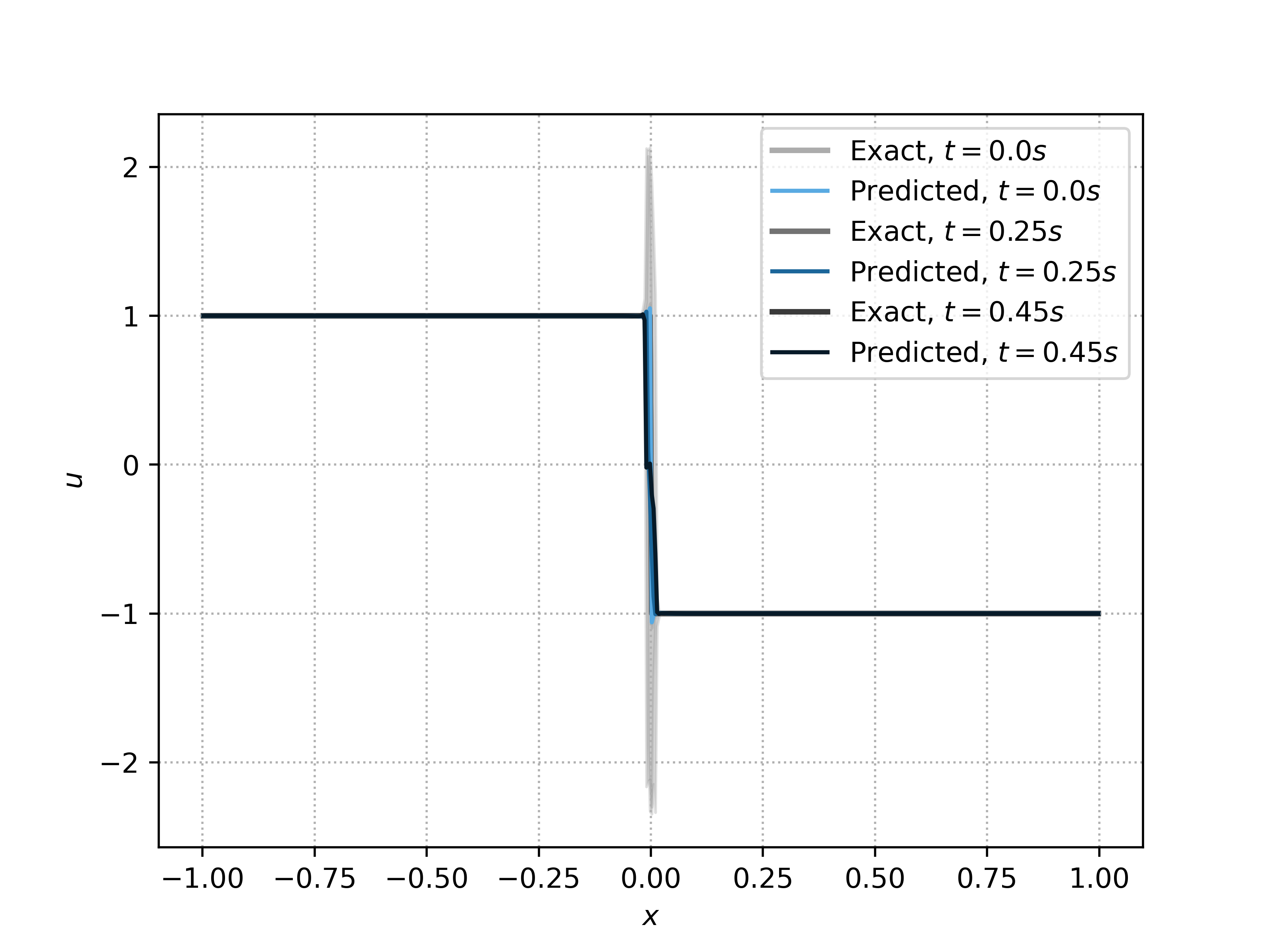}}
        \caption{Standing Shock Solution , $\cE^T_r(\theta^\ast_\S) =0.01$}
        \label{fig:standing}
    \end{subfigure}
    \begin{subfigure}{.48\textwidth}
        \centering
        \includegraphics[width=1\linewidth]{{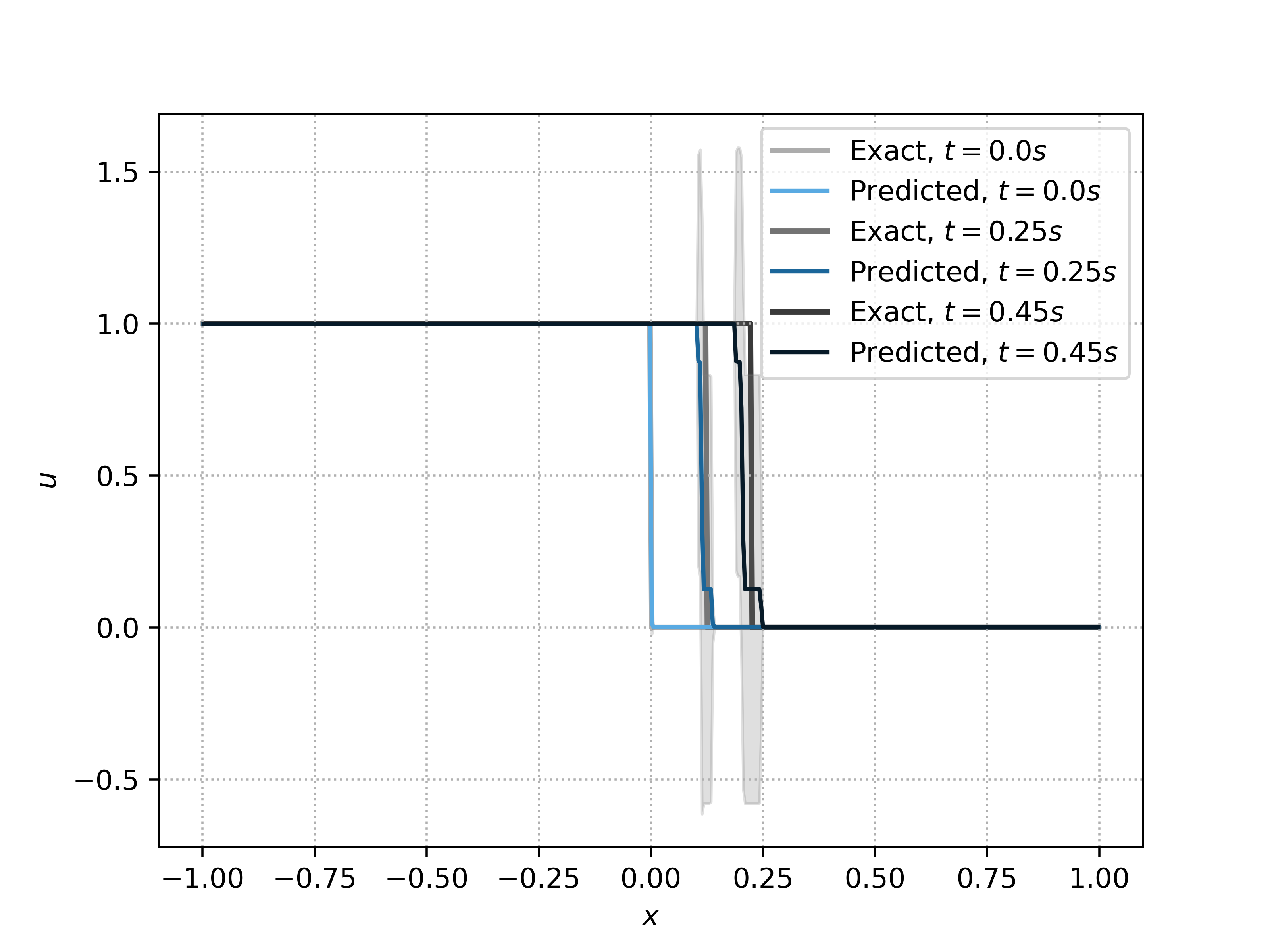}}
        \caption{Moving Shock Solution, $\cE^T_r(\theta^\ast_\S) =0.019$}
        \label{fig:moving}
    \end{subfigure}
    \begin{subfigure}{.48\textwidth}
        \centering
        \includegraphics[width=1\linewidth]{{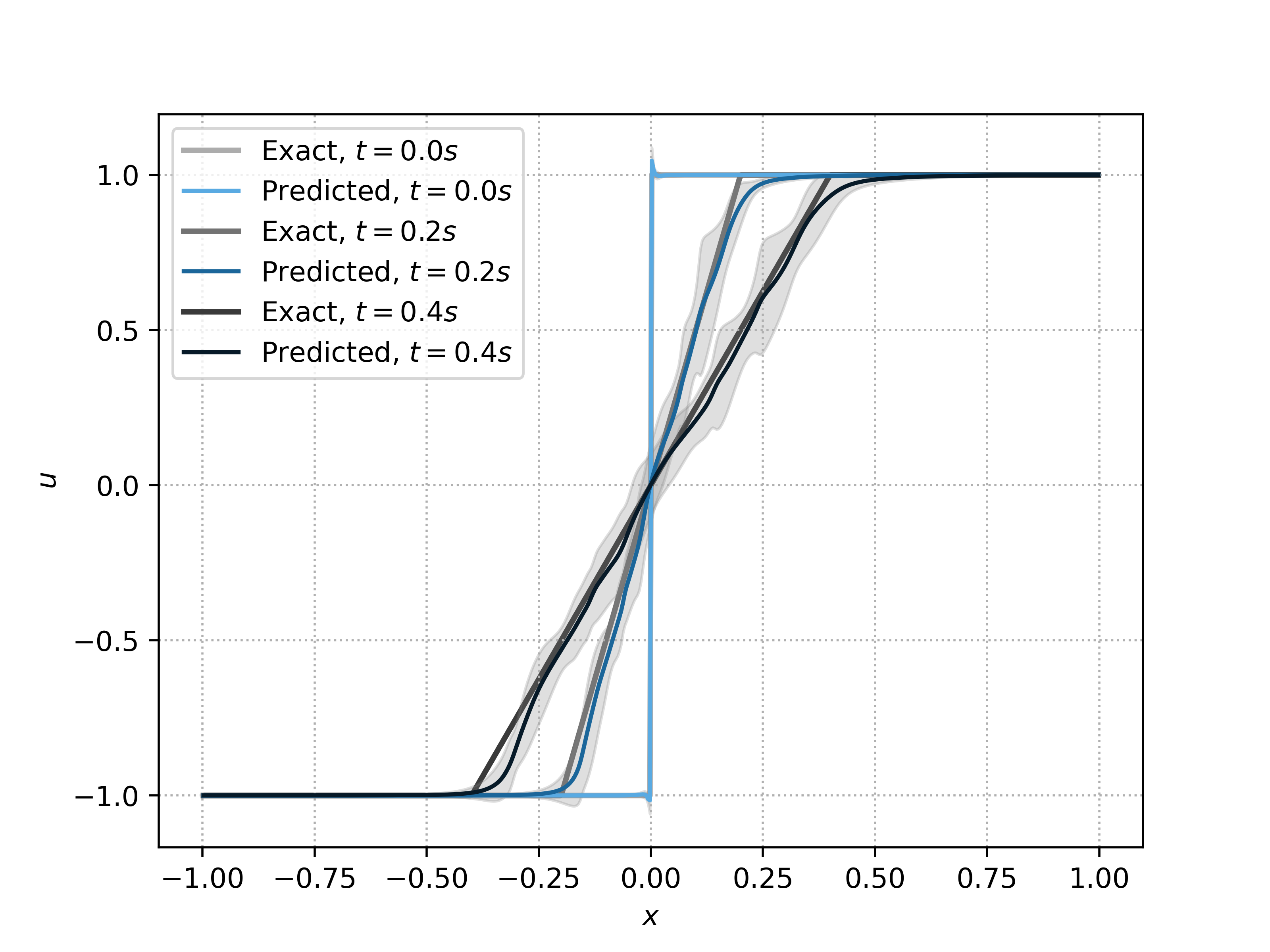}}
        \caption{Rarefaction Wave, $\cE^T_r(\theta^\ast_\S) =0.022$}
        \label{fig:rarefaction}
    \end{subfigure}
    \begin{subfigure}{.48\textwidth}
        \centering
        \includegraphics[width=1\linewidth]{{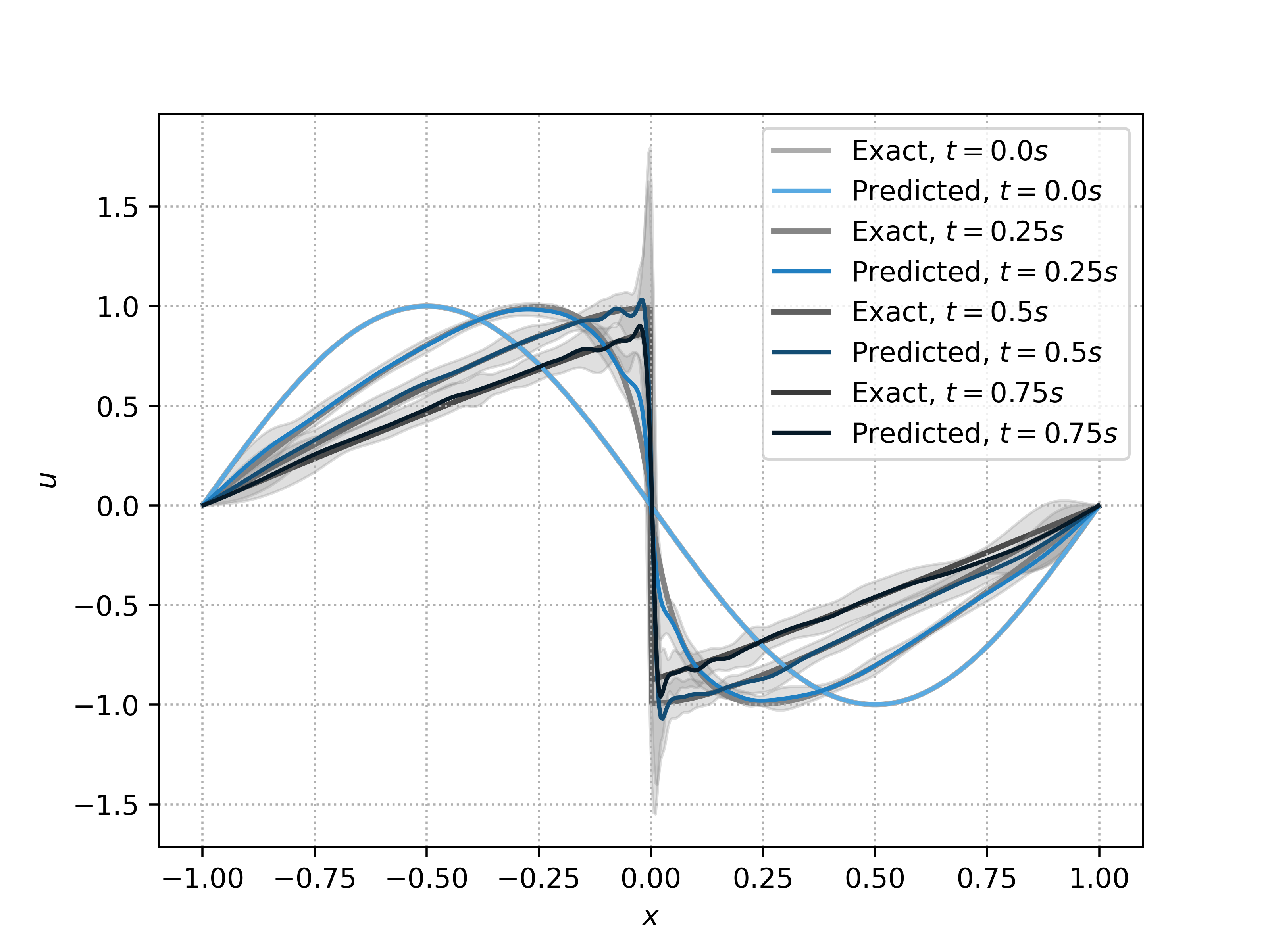}}
        \caption{Initial Sine Wave, $\cE^T_r(\theta^\ast_\S) =0.057$}
        \label{fig:sine}
    \end{subfigure}
    
    \caption{Exact solutions and average predictions obtained with \emph{wPINN} for the Burgers' equation.  Retraining average and standard deviation are plotted.}
\label{fig:burg}
\end{figure}

\subsection{Rarefaction Wave}
We further test the performance of \emph{wPINNs} by considering the Burgers' equation with initial data,
\begin{equation}
    u_0(x) = \begin{cases} 
      -1 & x\leq 0 \\
      1 & x> 0 
   \end{cases}
\end{equation}
The exact solution, given by,
\begin{equation}
   u(x, t) = \begin{cases} 
      -1 & x\leq -t \\
      \frac{x}{t} & -t < x\leq t \\
      1 & x> t
   \end{cases}
\end{equation}
corresponds to a \emph{rarefaction wave}. 
\begin{table}[htbp] 
    \centering
    \renewcommand{\arraystretch}{1.1} 
    \footnotesize{
        \begin{tabular}{ c c c c c c c  } 
            \toprule
            \bfseries   &\bfseries $M_{int}$   & \bfseries  $M_{tb}$   &\bfseries $M_{sb}$    & \bfseries $\cE_{r}$  & \bfseries $\cE^T_r$   \\ 
            \midrule
            \midrule
              \bfseries Standing Shock    & 16384   & 4096 & 4096 &0.005  &0.01  \\
            \midrule 
             \bfseries  Moving Shock   & 16384  & 4096 & 4096 &0.011 &0.019  \\
                \midrule 
              \bfseries Rarefaction Wave  & 16384  & 4096& 4096&0.013 &0.022 \\
              \midrule 
              \bfseries Initial Sine Wave   & 16384  & 4096& 4096&0.03 &0.057 \\
            \bottomrule
        \end{tabular}
    \caption{Number of training samples, total error (at final time) and total error over time, obtained with \emph{wPINNs} (average) predictions in the numerical experiments for the Burgers' equation}
    \label{tab:res}
    }
    \end{table} 
We observe that this initial datum is often used to illustrate the \emph{multiplicity} of weak solutions of hyperbolic conservation laws as the \emph{standing shock}, corresponding to the initial datum is clearly a weak solution but does not satisfy the entropy conditions. To illustrate the rationale behind considering \emph{entropy residuals} \eqref{eq:R-def} in our definitions of the loss function \eqref{eq:hlf} in the \emph{wPINNs} algorithm \ref{alg:wpinn_reset}, we first run the same algorithm but replace the entropy residual \eqref{eq:R-def} in Algorithm \ref{alg:wpinn_reset} with the following residuals,

\begin{figure}
    \begin{subfigure}{.48\textwidth}
        \centering
        \includegraphics[width=1\linewidth]{{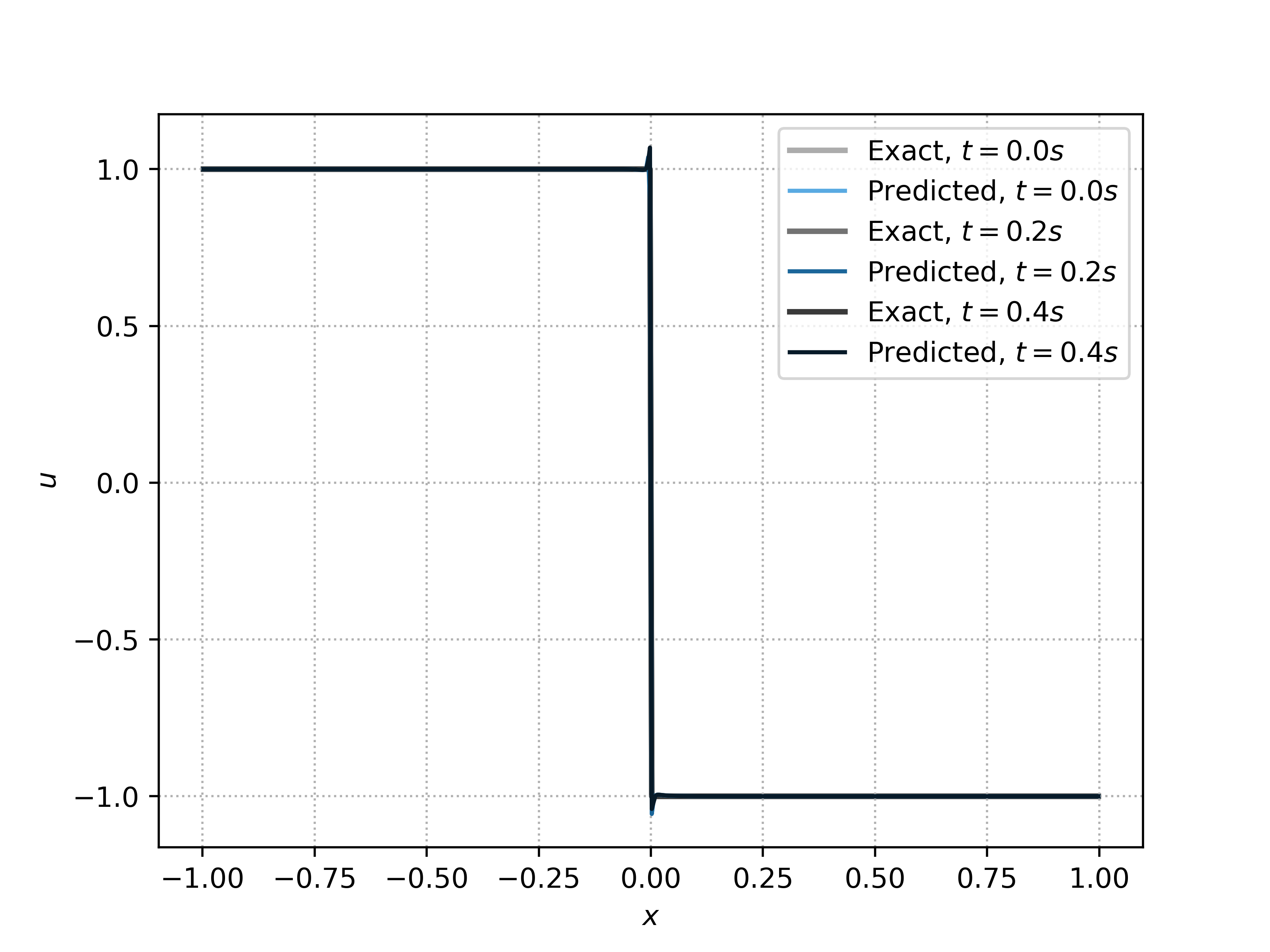}}
        \caption{Standing Shock Solution , $\cE^T_r(\theta^\ast_\S) = 0.0004$}
        \label{fig:standing_best}
    \end{subfigure}
    \begin{subfigure}{.48\textwidth}
        \centering
        \includegraphics[width=1\linewidth]{{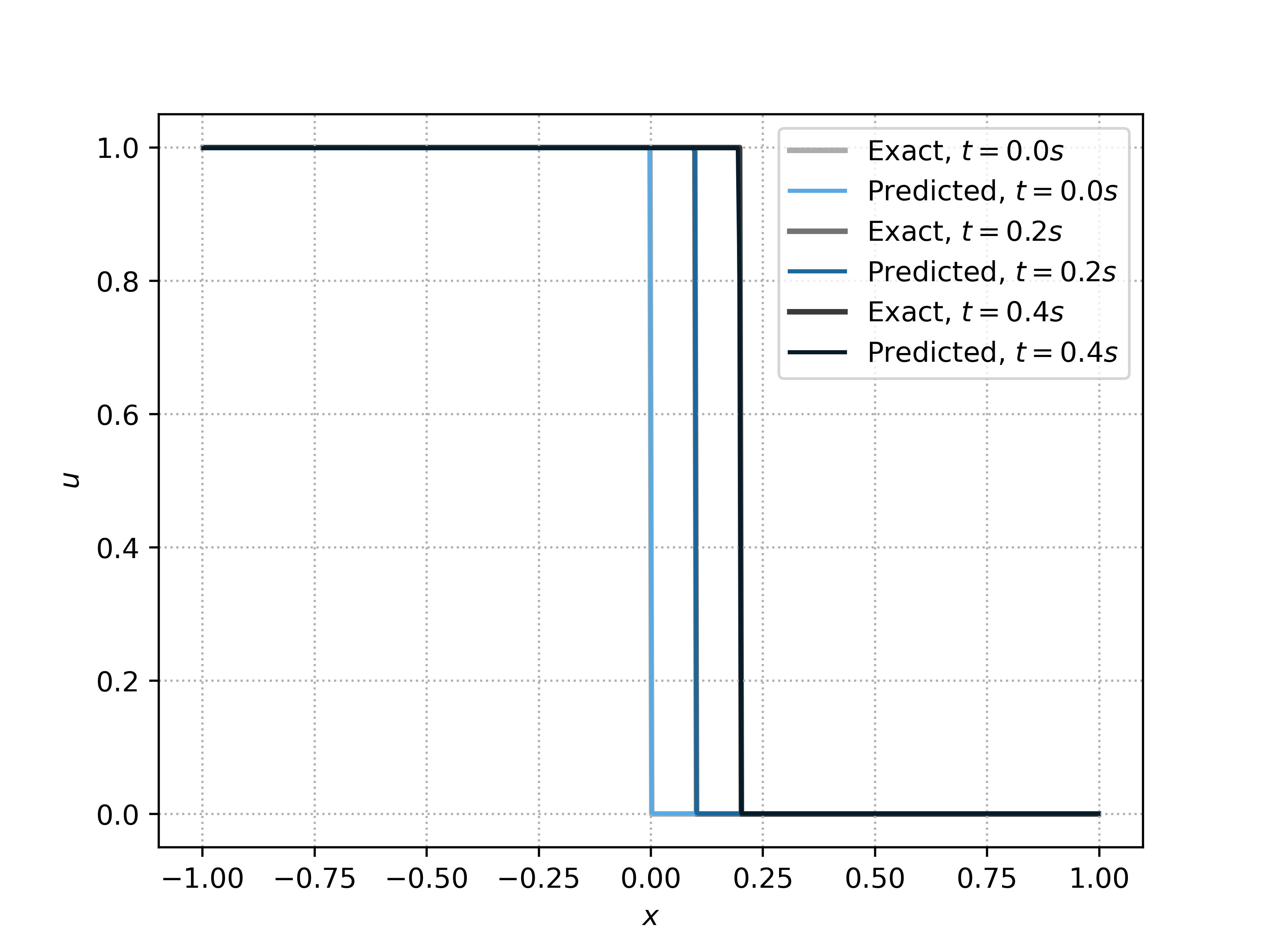}}
        \caption{Moving Shock Solution, $\cE^T_r(\theta^\ast_\S) =0.002$}
        \label{fig:moving_best}
    \end{subfigure}
    \begin{subfigure}{.48\textwidth}
        \centering
        \includegraphics[width=1\linewidth]{{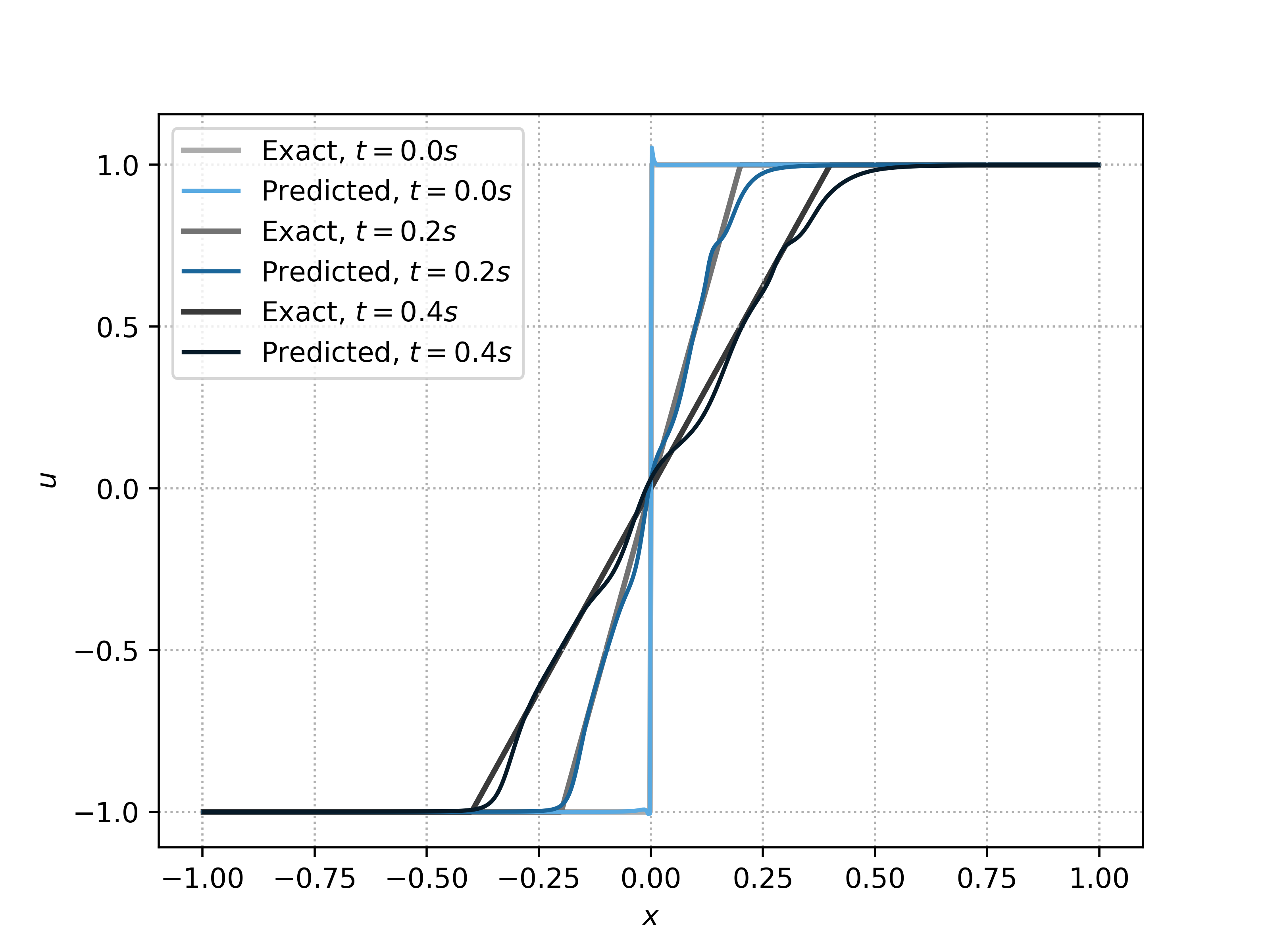}}
        \caption{Rarefaction Wave, $\cE^T_r(\theta^\ast_\S) =0.019$}
        \label{fig:rarefaction_best}
    \end{subfigure}
    \begin{subfigure}{.48\textwidth}
        \centering
        \includegraphics[width=1\linewidth]{{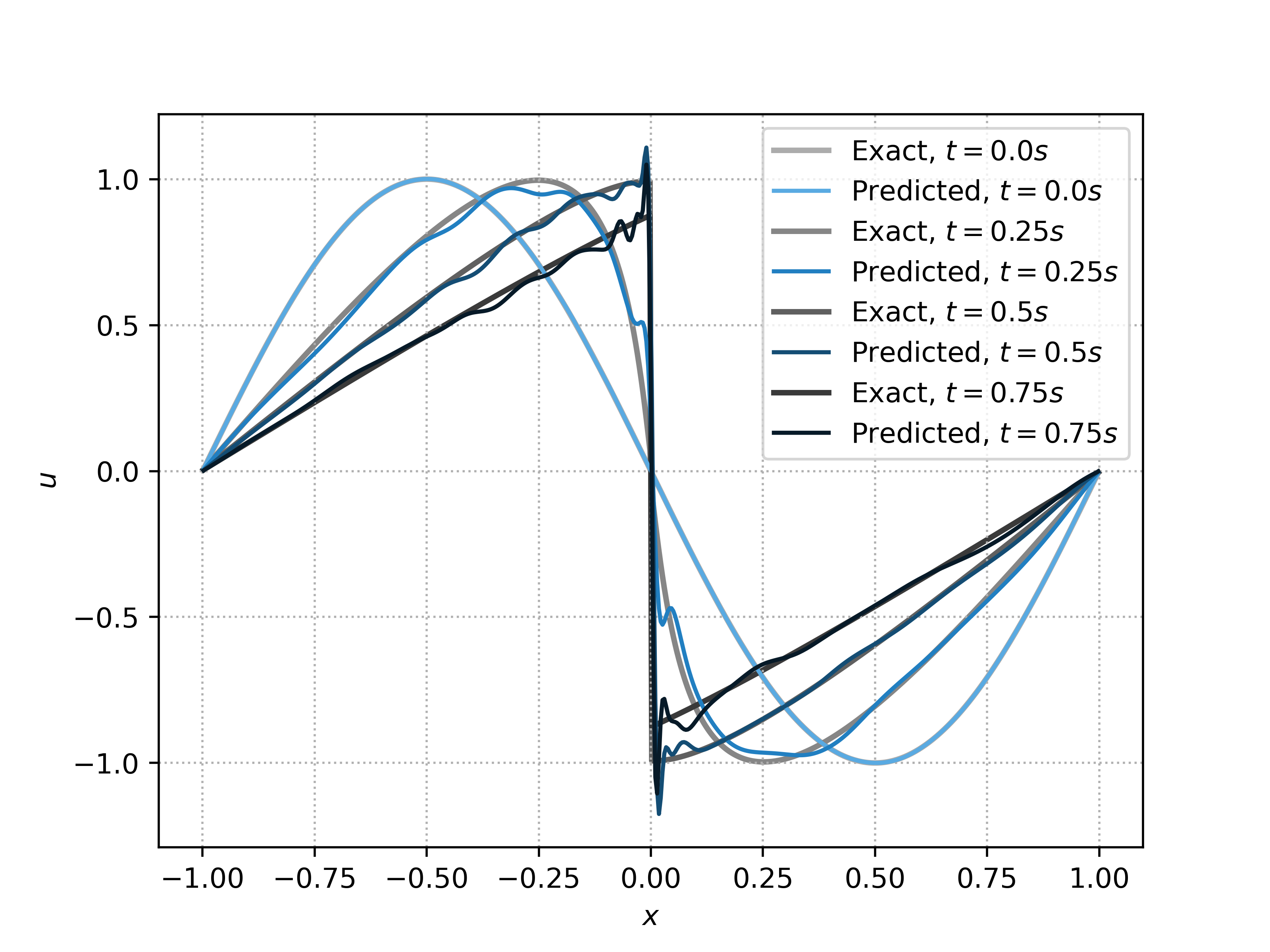}}
        \caption{Initial Sine Wave, $\cE^T_r(\theta^\ast_\S) =0.047$}
        \label{fig:sine_besz}
    \end{subfigure}
    
    \caption{Exact solutions and best predictions obtained with \emph{wPINN} for the Burgers' equation.}
\label{fig:burg_best_overall}
\end{figure}

\begin{equation}\label{eq:weakres_burg}
   r_{int, \theta, \eta} = \int_{D}\int_{[0,T]}\Big( u_{\theta, t}(x,t)\varphi_\eta(x,t) - f(u_\theta(x,t)) \varphi_{\eta, x}(x,t)\Big)dtdx  
\end{equation}
and 
\begin{equation}
    J_{pde}(\theta, \eta) = \frac{\Big(\sum_{m=1}^{M_{int}} r_{int,\theta, \eta}(y_m, c)\Big)^2}{\sum_{m=1}^{M_{int}} \partial_x \varphi_\eta(y_m)^2},
\end{equation} 
that correspond to the standard weak formulation (see definition \ref{def:weak-solution}) of the scalar conservation law. The resulting predictions are plotted in figure \ref{fig:naive_rar} and show that the resulting \emph{wPINN} only approximated the \emph{non-entropic} standing shock solution corresponding to the initial datum. Thus, a naive weak formulation of PINNs does not suffice in accurate approximations of scalar conservation laws. 
\begin{figure}
        \centering
        \includegraphics[width=0.5\linewidth]{{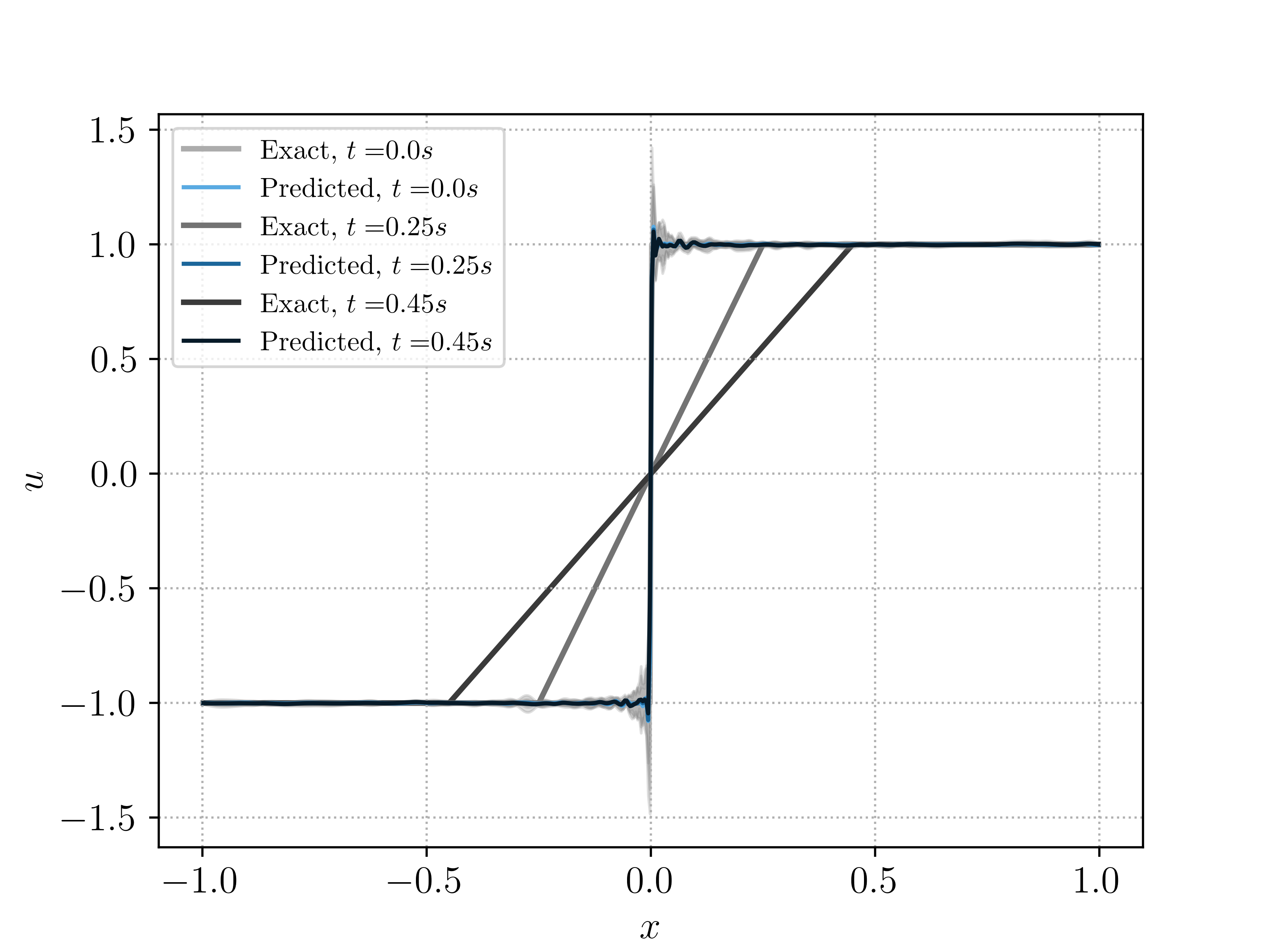}}
    \caption{Exact rarefaction wave solution and prediction obtained with weak PINNs, without entropy conditions incorporated into the weak residual.}
\label{fig:naive_rar}
\end{figure}
On the other hand, the \emph{wPINN} average predictions of algorithm \ref{alg:wpinn_reset}, with the entropy residual \eqref{eq:R-def}, provide accurate approximation of the rarefaction wave entropy solution, as shown in Figure  \ref{fig:rarefaction} and Table  \ref{tab:res}. In particular, the error at the final time $T$ is approximately $2 \%$ on average, whereas the error corresponding to the best hyperparameter configuration (see Figure \ref{fig:burg_best_overall}) is only slightly lower ($1.9 \%$). 
\subsection{Sine Wave Initial Datum}
As a final numerical experiment, we consider the Burgers' equation with the initial data,
$$
{u}_0(x) = -\sin(\pi x)$$ 
and zero Dirichlet boundary conditions in the spatio-temporal domain $[-1,1]\times[0,1]$. The exact solution (approximated in Figure \ref{fig:sine} with a high-resolution finite volume scheme), shows a complex evolution with both steepening as well as expansions of the sine wave that eventually form into a shock wave that separates two rarefactions. We run the \emph{wPINNs} algorithm \ref{alg:wpinn_reset}, with low-discrepancy Sobol points \cite{mishra2021enhancing}, instead of random collocation points. An ensemble training procedure, based on the hyperparameters presented in Table \ref{tab:sine_hyp}. The remaining parameters are set as follows:  $L_\theta=4$, $L_\eta=2$, $l_\theta=20$, $l_\eta=10$,  $\tau_\eta = 0.015$, $\tau_\theta = 0.01$ and $\sin$ activation function for both the networks. The networks are trained for $e=75000$ epochs and parameters reinitialized $n_\theta=15$ times, on account of the more complex underlying solution.

The (average) predictions with the \emph{wPINNs} algorithm \ref{alg:wpinn_reset} are depicted in Figure \ref{fig:sine} and show that this complicated underlying solution is approximated accurately with \emph{wPINNs}, although there are very small spurious oscillations in the approximation. These might be further eliminated by adding additional regularization terms, such as on the BV-norm into the loss function \eqref{eq:hlf}. Nevertheless, as shown in Table \ref{tab:res}, the error  over the entire time period is approximately $3\%$, whereas the error at final time is understandably higher. This should also be contrasted with the relative error of approximately $24\%$, obtained for this particular test case with conventional PINNs, as observed in \cite{MM1} (Figure 6 (d)). Moreover, the error $\cE_r^T$ is even smaller and the approximation significantly more accurate, with the best performing hyperparameter configuration shown in Figure \ref{fig:burg_best_overall}.  
\section{Discussion}
Physics informed neural networks (PINNs) have been extremely successful in approximating both forward and inverse problems, in an unsupervised manner, for a very diverse set of PDEs. Recent theoretical work on PINNs, for instance in \cite{MM1,deryck2021approximation}, suggest that \emph{regularity} of solutions of the underlying PDE is a key requirement in the derivation of estimates on PINN error. In particular, this analysis suggests that the conventional form of PINNs can fail at accurately approximating nonlinear PDEs such as hyperbolic conservation laws, whose solutions are not sufficiently regular. This was further verified in numerical experiments, such as the ones presented in \cite{MM1}. 

Our main aim in this paper was to design a novel variant of PINNs for the accurate approximation of entropy solutions of scalar conservation laws. To this end, 
\begin{itemize}
    \item We base the PDE residual on the weak form of the Kruzkhov entropy conditions \eqref{eq:R-def} and solve the resulting min-max optimization problem for determining parameters (weights and biases), for the neural networks approximating the entropy solution as well as test functions. The resulting PINN is termed as a weak PINN (\emph{wPINN}).
    \item We prove rigorous bounds on different sources of error associated with \emph{wPINNs} in Section \ref{sec:3} to show that the resulting errors can be made arbitrarily small by ensuring a small enough training error, choosing neural networks of suitable size and enoough (random) collocation points.
    \item We present numerical experiments with the Burgers' equation to illustrate that \emph{wPINNs}, with suitable choices of loss functions and training protocol (see section \ref{sec:weak-pinn-training}), can approximate the entropy solutions of scalar conservation laws accurately.
    
\end{itemize}
Thus, we propose a novel unsupervised learning algorithm for approximating scalar conservation laws and show, both theoretically as well as empirically, that it provides an accurate approximation to entropy solutions. Given that, the algorithm was fully unsupervised i.e., no labelled data (solution values in the interior of the space-time domain) are used, \emph{wPINNs} can serve as an alternative to existing high-resolution finite volume, finite difference and discontinuous Galerkin finite element methods for approximating conservation laws. 

Below, we discuss possible shortcomings and extensions of \emph{wPINNs}, 
\begin{itemize}
    \item In comparison to conventional PINNs, \emph{wPINNs} entail the (approximate) solution of a min-max optimization problem. Thus, training \emph{wPINNs} is clearly more computationally expensive than training PINNs as only a minimization problem is solved in the latter. However, given that conventional PINNs can fail to accurately approximate discontinuous solutions of conservation laws, it is imperative to use \emph{wPINNs} in this context. We aim to explore faster training algorithms for solving the min-max optimization problem and plan to take inspiration from the extensive literature on training GANs in this regard. \item A key advantage of machine learning approaches such as conventional PINNs and \emph{wPINNs} is their ability to serve as a fast surrogate, particularly for parametric PDEs (see \cite{BKMM1} for an example of PINNs solving a parametrized KdV equation). Once the \emph{wPINN} has been trained for a (random) sampling of the parameter space, it can infer solutions with respect to other parameters at practically zero computational cost. Moreover, one can also use \emph{wPINNs} within a physics informed operator learning framework \cite{PINO,deryck2022generic} to approximate the semi-group of the underlying solution operator. The use of \emph{wPINNs} in the context of parametric PDEs and operator learning will be considered in future work. 
    \item Lastly, we presented the algorithm and provided error estimates for scalar conservation laws in one space dimension. The extension of the algorithm and estimates to multi-dimensional scalar conservation laws is straightforward as Kruzkhov entropies are well-defined in this case. The extension of \emph{wPINNs} to systems of conservation laws is non-trivial. In particular, there is no analogue of (infinitely many) Kruzkhov entropies for systems. Rather, one has to work with, usually, a single entropy family (for instance the thermodynamic entropy for the compressible Euler equations or total energy for shallow-water equations). Adding such entropies to the residual is straightforward. However, it is unclear if a single entropy will drive the algorithm towards a physically relevant solution. This extension will also be considered in the future.

\end{itemize}

\section*{Acknowledgments}\label{sec:Acknowledgements}
The research of RM and SM was partly performed under a project that has received funding from the European Research Council (ERC) under the European Union’s Horizon 2020 research and innovation programme (grant agreement No. 770880). The authors thank Prof. Ulrik S. Fjordholm (University of Oslo, Norway) and Prof. Ujjwal Koley (TIFR-CAM, Bangalore, India) for insightful discussions.

\bibliographystyle{abbrv}
\bibliography{ref}

\begin{thebibliography}{10}

\bibitem{wgan}
M.~Arjovsky, S.~Chintala, and L.~Bottou.
\newblock Wasserstein gan.
\newblock {\em arXiv preprint arXiv:1701.07875v3}, 2017.

\bibitem{BKMM1}
G.~Bai, U.~Koley, S.~Mishra, and R.~Molinaro.
\newblock Physics informed neural networks ({PINNs}) for approximating
  nonlinear dispersive {PDEs}.
\newblock {\em arXiv preprint arXiv:2104.05584}, 2021.

\bibitem{bartels2012total}
S.~Bartels.
\newblock Total variation minimization with finite elements: convergence and
  iterative solution.
\newblock {\em SIAM Journal on Numerical Analysis}, 50(3):1162--1180, 2012.

\bibitem{cuomo}
S.~Cuomo, V.~S. Di~Cola, F.~Giampaolo, G.~Rozza, M.~Raissi, and F.~Piccialli.
\newblock Scientific machine learning through physics-informed neural networks:
  Where we are and what's next.
\newblock {\em arXiv preprint arXiv:2201.05624}, 2022.

\bibitem{deryck2021navierstokes}
T.~De~Ryck, A.~D. Jagtap, and S.~Mishra.
\newblock Error analysis for {PINNs} approximating the {Navier-Stokes}
  equations.
\newblock {\em arXiv preprint arXiv:2203.09346}, 2022.

\bibitem{deryck2021approximation}
T.~{De Ryck}, S.~Lanthaler, and S.~Mishra.
\newblock On the approximation of functions by tanh neural networks.
\newblock {\em Neural Networks}, 2021.

\bibitem{deryck2021pinn}
T.~De~Ryck and S.~Mishra.
\newblock Error analysis for physics informed neural networks {(PINNs)}
  approximating {Kolmogorov PDEs}.
\newblock {\em arXiv preprint arXiv:2106.14473}, 2021.

\bibitem{deryck2022generic}
T.~De~Ryck and S.~Mishra.
\newblock Generic bounds on the approximation error for physics-informed (and)
  operator learning.
\newblock {\em arXiv preprint arXiv:2205.11393}, 2022.

\bibitem{DPT}
M.~Dissanayake and N.~Phan-Thien.
\newblock Neural-network-based approximations for solving partial diﬀerential
  equations.
\newblock {\em Communications in Numerical Methods in Engineering}, 1994.

\bibitem{HEJ1}
W.~E, J.~Han, and A.~Jentzen.
\newblock Deep learning-based numerical methods for high-dimensional parabolic
  partial differential equations and backward stochastic differential
  equations.
\newblock {\em Communications in Mathematics and Statistics}, 5(4):349--380,
  2017.

\bibitem{gan}
I.~Goodfellow, J.~Pouget-Abadie, M.~Mirza, B.~Xu, D.~Warde-Farley, S.~Ozair,
  A.~Courville, and Y.~Bengio.
\newblock Generative adversarial networks.
\newblock {\em arXiv preprint arXiv:}, 2014.

\bibitem{holden2015front}
H.~Holden and N.~H. Risebro.
\newblock {\em Front tracking for hyperbolic conservation laws}, volume 152.
\newblock Springer, 2015.

\bibitem{hu2021extended}
Z.~Hu, A.~D. Jagtap, G.~E. Karniadakis, and K.~Kawaguchi.
\newblock When do extended physics-informed neural networks ({XPINNs}) improve
  generalization?
\newblock {\em arXiv preprint arXiv:2109.09444}, 2021.

\bibitem{jagtap2020extended}
A.~D. Jagtap and G.~E. Karniadakis.
\newblock Extended physics-informed neural networks ({XPINNs}): A generalized
  space-time domain decomposition based deep learning framework for nonlinear
  partial differential equations.
\newblock {\em Communications in Computational Physics}, 28(5):2002--2041,
  2020.

\bibitem{jag2}
A.~D. Jagtap, E.~Kharazmi, and G.~E. Karniadakis.
\newblock Conservative physics-informed neural networks on discrete domains for
  conservation laws: Applications to forward and inverse problems.
\newblock {\em Computer Methods in Applied Mechanics and Engineering},
  365:113028, 2020.

\bibitem{jagtap2022physics}
A.~D. Jagtap, Z.~Mao, N.~Adams, and G.~E. Karniadakis.
\newblock Physics-informed neural networks for inverse problems in supersonic
  flows.
\newblock {\em arXiv preprint arXiv:2202.11821}, 2022.

\bibitem{jagtap2022deep}
A.~D. Jagtap, D.~Mitsotakis, and G.~E. Karniadakis.
\newblock Deep learning of inverse water waves problems using multi-fidelity
  data: Application to {Serre--Green--Naghdi} equations.
\newblock {\em Ocean Engineering}, 248:110775, 2022.

\bibitem{jin2021nsfnets}
X.~Jin, S.~Cai, H.~Li, and G.~E. Karniadakis.
\newblock {NSFnets (Navier-Stokes flow nets): Physics-informed} neural networks
  for the incompressible {Navier-Stokes} equations.
\newblock {\em Journal of Computational Physics}, 426:109951, 2021.

\bibitem{vpinns}
E.~Kharazmi, Z.~Zhang, and G.~E. Karniadakis.
\newblock Variational physics informed neural networks for solving partial
  differential equations.
\newblock {\em arXiv preprint arXiv:1912.00873}, 2019.

\bibitem{Kuty}
G.~Kutyniok, P.~Petersen, M.~Raslan, and R.~Schneider.
\newblock A theoretical analysis of deep neural networks and parametric pdes.
\newblock {\em Constructive Approximation}, pages 1--53, 2021.

\bibitem{Lag2}
I.~E. Lagaris, A.~Likas, and P.~G. D.
\newblock Neural-network methods for boundary value problems with irregular
  boundaries.
\newblock {\em IEEE Transactions on Neural Networks}, 11:1041--1049, 2000.

\bibitem{Lag1}
I.~E. Lagaris, A.~Likas, and D.~I. Fotiadis.
\newblock Artificial neural networks for solving ordinary and partial
  differential equations.
\newblock {\em IEEE Transactions on Neural Networks}, 9(5):987--1000, 2000.

\bibitem{DLnat}
Y.~LeCun, Y.~Bengio, and G.~Hinton.
\newblock Deep learning.
\newblock {\em Nature}, 521(7553):436--444, 2015.

\bibitem{FNO}
Z.~Li, N.~Kovachki, K.~Azizzadenesheli, B.~Liu, K.~Bhattacharya, A.~Stuart, and
  A.~Anandkumar.
\newblock Fourier neural operator for parametric partial differential
  equations, 2020.

\bibitem{deeponets}
L.~Lu, P.~Jin, and G.~E. Karniadakis.
\newblock {DeepONet}: Learning nonlinear operators for identifying differential
  equations based on the universal approximation theorem of operators.
\newblock {\em arXiv preprint arXiv:1910.03193}, 2019.

\bibitem{LMR1}
K.~O. Lye, S.~Mishra, and D.~Ray.
\newblock Deep learning observables in computational fluid dynamics.
\newblock {\em Journal of Computational Physics}, page 109339, 2020.

\bibitem{LMPR1}
K.~O. Lye, S.~Mishra, D.~Ray, and P.~Chandrashekar.
\newblock Iterative surrogate model optimization ({ISMO}): An active learning
  algorithm for pde constrained optimization with deep neural networks.
\newblock {\em Computer Methods in Applied Mechanics and Engineering},
  374:113575, 2021.

\bibitem{KAR6}
Z.~Mao, A.~D. Jagtap, and G.~E. Karniadakis.
\newblock Physics-informed neural networks for high-speed flows.
\newblock {\em Computer Methods in Applied Mechanics and Engineering},
  360:112789, 2020.

\bibitem{MM2}
S.~Mishra and R.~Molinaro.
\newblock Estimates on the generalization error of physics-informed neural
  networks for approximating a class of inverse problems for pdes.
\newblock {\em IMA Journal of Numerical Analysis}, 2021.

\bibitem{MM3}
S.~Mishra and R.~Molinaro.
\newblock Physics informed neural networks for simulating radiative transfer.
\newblock {\em Journal of Quantitative Spectroscopy and Radiative Transfer},
  270:107705, 2021.

\bibitem{MM1}
S.~Mishra and R.~Molinaro.
\newblock {Estimates on the generalization error of physics-informed neural
  networks for approximating PDEs}.
\newblock {\em IMA Journal of Numerical Analysis}, 01 2022.
\newblock drab093.

\bibitem{mishra2021enhancing}
S.~Mishra and T.~K. Rusch.
\newblock Enhancing accuracy of deep learning algorithms by training with
  low-discrepancy sequences.
\newblock {\em SIAM Journal on Numerical Analysis}, 59(3):1811--1834, 2021.

\bibitem{KAR7}
G.~Pang, L.~Lu, and G.~E. Karniadakis.
\newblock {fPINNs}: Fractional physics-informed neural networks.
\newblock {\em SIAM journal of Scientific computing}, 41:A2603--A2626, 2019.

\bibitem{petersen2018optimal}
P.~Petersen and F.~Voigtlaender.
\newblock Optimal approximation of piecewise smooth functions using deep relu
  neural networks.
\newblock {\em Neural Networks}, 108:296--330, 2018.

\bibitem{KAR1}
M.~Raissi and G.~E. Karniadakis.
\newblock Hidden physics models: Machine learning of nonlinear partial
  differential equations.
\newblock {\em Journal of Computational Physics}, 357:125--141, 2018.

\bibitem{KAR2}
M.~Raissi, P.~Perdikaris, and G.~E. Karniadakis.
\newblock Physics-informed neural networks: A deep learning framework for
  solving forward and inverse problems involving nonlinear partial differential
  equations.
\newblock {\em Journal of Computational Physics}, 378:686--707, 2019.

\bibitem{KAR4}
M.~Raissi, A.~Yazdani, and G.~E. Karniadakis.
\newblock Hidden fluid mechanics: A {Navier-Stokes} informed deep learning
  framework for assimilating flow visualization data.
\newblock {\em arXiv preprint arXiv:1808.04327}, 2018.

\bibitem{SZ1}
C.~Schwab and J.~Zech.
\newblock Deep learning in high dimension: Neural network expression rates for
  generalized polynomial chaos expansions in uq.
\newblock {\em Analysis and Applications}, 17(01):19--55, 2019.

\bibitem{shin2020convergence}
Y.~Shin, J.~Darbon, and G.~E. Karniadakis.
\newblock On the convergence and generalization of physics informed neural
  networks.
\newblock {\em arXiv preprint arXiv:2004.01806}, 2020.

\bibitem{Zhang1}
Y.~Shin, Z.~Zhang, and G.~E. Karniadakis.
\newblock Error estimates of residual minimization using neural networks for
  linear equations.
\newblock {\em arXiv preprint arXiv:2010.08019}, 2020.

\bibitem{shukla2021physics}
K.~Shukla, A.~D. Jagtap, J.~L. Blackshire, D.~Sparkman, and G.~E. Karniadakis.
\newblock A physics-informed neural network for quantifying the microstructural
  properties of polycrystalline nickel using ultrasound data: A promising
  approach for solving inverse problems.
\newblock {\em IEEE Signal Processing Magazine}, 39(1):68--77, 2021.

\bibitem{shukla2021parallel}
K.~Shukla, A.~D. Jagtap, and G.~E. Karniadakis.
\newblock Parallel physics-informed neural networks via domain decomposition.
\newblock {\em Journal of Computational Physics}, 447:110683, 2021.

\bibitem{PINO}
S.~Wang and P.~Perdikaris.
\newblock Long-time integration of parametric evolution equations with
  physics-informed {DeepONets}.
\newblock {\em arXiv preprint arXiv:2106.05384}, 2021.

\bibitem{KAR8}
L.~Yang, X.~Meng, and G.~E. Karniadakis.
\newblock {B-PINNs}: Bayesian physics-informed neural networks for forward and
  inverse pde problems with noisy data.
\newblock {\em Journal of Computational Physics}, 425:109913, 2021.

\end{thebibliography}

\appendix

\section{Auxiliary results}

\subsection{Auxiliary results for Section \ref{sec:32}}

We will use the partition of unity construction using tanh neural networks from \cite{deryck2021approximation}.

\begin{lemma}\label{lem:Phi}
Let $\epsilon>0$ and $1\leq p<\infty$. For any $\varphi\in\Phi_\epsilon$ it holds that $\abs{\varphi}_{W^{1,p}} =  \bigO(\beta^{1+2(p-1)/p})$, $p\in\N$, and $\abs{\varphi}_{W^{1,\infty}} =  \bigO(\beta^{3})$.
\end{lemma}
\begin{proof}
One can easily calculate that $\partial_t \overline{\varphi}_\epsilon = \alpha \overline{\varphi}_\epsilon +\beta \overline{\varphi}_\epsilon$ and $\partial_x \overline{\varphi}_\epsilon = \beta \overline{\varphi}_\epsilon$. Moreover, one can easily find that $\norm{\overline{\varphi}_\epsilon}_{L^\infty} = \bigO(\beta^{2})$ and $\norm{\overline{\varphi}_\epsilon}_{L^1} = \bigO(1)$. As a result, one can find that for $p\in\N$ it holds that $\abs{\overline{\varphi}_\epsilon}_{W^{1,p}} = \bigO(\beta^{2(p-1)/p})$. The bound from the statement follows immediately. 
\end{proof}

\begin{lemma}\label{lem:alpha-growth}
Let $N\in\N$ and $0<\epsilon<1$. If we set $\alpha = N\ln(N^2/\epsilon)$ then it holds that, 
\begin{align}\label{eq:alpha}
    \alpha/N \geq 1, \quad 1 - \sigma(\alpha/N) \leq \epsilon, \quad \alpha^m \abs{\sigma^{(m)}(\alpha/N)} \leq \epsilon \text{  for } m=1,2. 
\end{align}
\end{lemma}
\begin{proof}
The statement follows directly from \cite[Lemma B.4]{deryck2021navierstokes} by using that $4/e^2\leq 1$.
\end{proof}

The following auxiliary results are needed in the proof of Theorem \ref{thm:stability}.

\begin{lemma}\label{lem:chi}
Let $\delta, T>0$, $0<\delta<T/2$, $3\delta \leq z\leq T-3\delta$ and let $f\in  C^1([0,T]\setminus\{z\})$. Define $\chi_\delta(t) = \gamma(\sigma(\alpha (t-2\delta))-\sigma(\alpha (t-T+2\delta)))$ where $\alpha = \ln(1/\delta^3)/\delta$ and $\gamma^{-1} =\sigma(\alpha\delta)$. Then it holds that 
\begin{equation}
    \abs{\int_0^T f(t)\chi'_\delta(t) dt - f(2\delta)+f(T-2\delta)}\leq \left(T\norm{f}_{L^{\infty}([0,T])}+3 \ln(1/\delta)\norm{f'}_{L^{\infty}(B)}\right) \frac{4\delta}{1-\delta}, 
\end{equation}
where $B = [\delta, 3\delta]\cup[T-3\delta, T-\delta]$.
\end{lemma}

\begin{proof}

Define $\alpha = \ln(1/\delta^3)/\delta$ and $\gamma^{-1} = \alpha \int_{-\delta}^\delta\sigma'(\alpha s) ds = 2\sigma(\alpha \delta) \geq 2(1-\delta)$, where the last inequality follows from Lemma \ref{lem:alpha-growth} with $N=\epsilon \leftarrow \delta$. Assume that $z\geq 3\delta$. Taylor's theorem then guarantees the existence of $\zeta_t\in [\delta,3\delta]$ such that,
\begin{equation}
    \int_\delta^{3\delta} f(t) \gamma \alpha \sigma'(\alpha (t-2\delta)) dt = f(2\delta)\gamma \alpha \int_\delta^{3\delta}\sigma'(\alpha (t-2\delta)) dt + \gamma\alpha \int_\delta^{3\delta} f'(\zeta_t)(t-2\delta) \sigma'(\alpha (t-2\delta)) dt.
\end{equation}
From this, it follows using the definition of $\gamma$ that,
\begin{equation}
    \begin{split}
        \abs{\int_\delta^{3\delta} f(t) \gamma \alpha \sigma'(\alpha (t-2\delta)) dt-f(2\delta)} &\leq \gamma \alpha \delta^2 \norm{f'}_{L^\infty([\delta, 3\delta])}. 
    \end{split}
\end{equation}
Moreover, it follows from Lemma \ref{lem:alpha-growth} that,
\begin{equation}
    \abs{\int_{[0,\delta]\cup[3\delta,T]} f(t) \gamma \alpha \sigma'(\alpha (t-2\delta)) dt} \leq T\gamma \norm{f}_\infty \delta.
\end{equation}
Similarly we can find for $A = [T-3\delta, T-\delta]$ that, 
\begin{equation}
    \abs{\int_{0}^T f(t) \gamma \alpha \sigma'(\alpha (t-T+2\delta)) dt-f(T-2\delta)} \leq \gamma \alpha \delta^2 \norm{f'}_{L^\infty(A)} + T\gamma \norm{f}_\infty \delta.
\end{equation}
Conclude by combining all the previous inequalities.
\end{proof}

\begin{lemma}\label{lem:rho}
Let $0<\epsilon<\min\{1,b-a\}$ and $f\in  C^1([a,b])$. Define
\begin{equation}
    \rho_\epsilon(x) = \frac{\sigma(\beta(x+{\epsilon^2}))-\sigma(\beta(x-{\epsilon^2}))}{2\epsilon^2},
\end{equation}
where $\beta = -3\ln(\epsilon)/\epsilon$ and $\omega(t) = \sigma(\beta(t-\max\{a,t-\epsilon\}))-\sigma(\beta(t-\min\{b,t+\epsilon\}))$, for which it holds that $1-\epsilon \leq \omega(t) \leq 2$. Then it holds that,
\begin{equation}
    \abs{\int_a^b f(s) \rho_\epsilon(t-s)ds - \omega(t) f(t)} \leq 20\norm{f}_{C^{1}([a,b])} (b-a-\ln(\epsilon))\epsilon.
\end{equation}
\end{lemma}

\begin{proof}
First we observe that for any $\epsilon>0$ it holds that,
\begin{equation}\label{eq:rho-taylor}
\begin{split}
    &\abs{ \int_a^b f(s) \beta\sigma'(\beta(t-s))ds- \int_a^b f(s) \frac{\sigma(\beta(t-s+\epsilon^2))-\sigma(\beta(t-s-\epsilon^2))}{2\epsilon^2}ds} \\
    &\quad\leq 2(\epsilon^2\beta)^2\norm{\sigma'''}_\infty \norm{f}_{L^1} \leq 20(b-a) \norm{f}_{\infty}\epsilon, 
\end{split}
\end{equation}
where for the last inequality we used that $\norm{\sigma'''}_\infty\leq 2$ and $\epsilon \ln(1/\epsilon^3)^2\leq 5$ for $0<\epsilon<1$.

Next, we let $A = [a,b]\cap [t-\epsilon,t+\epsilon]$ and $B = [a,b]\setminus [t-\epsilon,t+\epsilon]$. Taylor's theorem then guarantees the existence of $\zeta_t\in A$ such that,
\begin{equation}
    \int_A f(s) \beta\sigma'(\beta(t-s))ds = f(t) \int_A \beta\sigma'(\beta(t-s))ds + \beta \int_A f'(\zeta_t)(s-t) \sigma'(\beta (t-s)) ds
\end{equation}
Using that if $\epsilon\leq b-a$ then, 
\begin{equation}
    2 \geq \omega(t) = \beta \int_A \sigma'(\beta (t-s)) ds  \geq \beta \int_0^{\epsilon}\sigma'(\beta s)ds = \sigma(\beta\epsilon)
    \geq 1-\epsilon,
\end{equation}
together with Lemma \ref{lem:alpha-growth} gives us,
\begin{equation}
    \begin{split}
        &\abs{\int_a^b f(s) \beta\sigma'(\beta(t-s))ds - \omega(t)f(t)}\\
        &\quad\leq \beta \int_A \abs{f'(\zeta_t)(s-t) \sigma'(\beta (t-s))} ds + \beta \int_B \abs{f(s) \sigma'(\beta (t-s))} ds\\
        &\quad\leq \frac{ \beta  \norm{f'}_{L^\infty(A)}}{N^2} +  (b-a) \norm{f}_{\infty} \epsilon\\
        &\quad\leq \norm{f}_{C^{1}([a,b])}(b-a+\ln(1/\epsilon^3))\epsilon.
    \end{split}
\end{equation}
\end{proof}

\begin{lemma}\label{lem:smooth-abs}
Let $\eta>0$. Define the function $\abs{\cdot}_\eta: \R\to [0,\infty): x\mapsto \abs{x}_\eta = \sqrt{x^2+\eta^2}$. It holds that
\begin{equation}\label{eq:abs-eta}
    0\leq \abs{x}_\eta-\abs{x} \leq \eta, \qquad \abs{\frac{d}{dx}\abs{x}_\eta}\leq 1.
\end{equation}
\end{lemma}
\begin{proof}
The first set of inequalities follows from 
\begin{equation}
    0\leq (\abs{x}_\eta-\abs{x})(\abs{x}_\eta+\abs{x}) = \abs{x}^2_\eta-\abs{x}^2 = \eta^2 \qquad \text{and}\qquad \abs{x}_\eta+\abs{x} \geq \eta,
\end{equation}
and the other inequality follows from $\abs{\frac{d}{dx}\abs{x}_\eta} = \frac{\abs{x}}{\sqrt{x^2+\eta^2}} \leq 1. $
\end{proof}

\begin{lemma}\label{lem:partial-int-t}
Let $B>0$, $0<\epsilon<1$, $T>4\epsilon$, $z\in W^{1,\infty}(([0,1]\times[0,T])^2)$ with $\abs{z(x,t,y,s)}\leq B$ for all $(x,t), (y,s)\in [0,1]\times[0,T]$ and $\overline{\varphi}_\epsilon^{y,s}$, $(y,s)\in [0,1]\times[0,T]$, as defined in \eqref{eq:varphi}. There exists an constant $C>0$ (independent of $z$ and $\epsilon$) such that,
\begin{equation}
    \abs{\int_0^1\int_0^T\int_0^1\int_0^T\left(\overline{\varphi}_\epsilon^{y,s}(x,t) \partial_t z(x,t,y,s) + z(x,t,y,s)\partial_t \overline{\varphi}_\epsilon^{y,s}(x,t) \right)dt dx ds dy} \leq CB\epsilon.
\end{equation}
\end{lemma}

\begin{proof}
Using integration by parts and Fubini's theorem we find that,
\begin{equation}\label{eq:integration-parts}
\begin{split}
    &\int_0^1\int_0^T\int_0^1\int_0^T\left(\overline{\varphi}_\epsilon^{y,s}(x,t) \partial_t z(x,t) + z(x,t)\partial_t \overline{\varphi}_\epsilon^{y,s}(x,t) \right)dt dx ds dy\\
    &\quad = \int_0^1\int_0^1\int_0^T \left(z(x,T)\overline{\varphi}_\epsilon^{y,s}(x,T)   - z(x,0) \overline{\varphi}_\epsilon^{y,s}(x,0) \right) dsdxdy.
\end{split}
\end{equation}
We will now bound the absolute value of both terms on the RHS of the above equation. Both terms can be bounded in a similar way, we only show the proof for the upper bound of the second term. Observe that $\chi_\epsilon$ is increasing on $[0,\epsilon]$ and $\rho_\epsilon$ is decreasing on $[0,T]$. Moreover, it holds that $\overline{\varphi}_\epsilon^{y,s}\geq 0$. 
Using this information we find for any $x,y\in[0,1]$ that,
\begin{equation}\label{eq:zx0}
    \begin{split}
        \abs{\int_0^T z(x,0) \overline{\varphi}_\epsilon^{y,s}(x,0)ds} &\leq B\rho_\epsilon(x-y) \int_0^T \chi_\epsilon\left(\frac{s}{2}\right)\rho_\epsilon(s)ds \\
        &\leq B\rho_\epsilon(x-y)\left[\chi_\epsilon\left(\epsilon\right)\int_0^\epsilon \rho_\epsilon(s)ds+\rho_\epsilon(\epsilon)\int_\epsilon^T \chi_\epsilon\left(\frac{s}{2}\right)ds\right].
    \end{split}
\end{equation}
From Lemma \ref{lem:chi} and Lemma \ref{lem:rho} it follows that the two integrals above are at most 2. Furthermore if $T>4\epsilon$ it holds that,
\begin{equation}
    \chi_\epsilon\left(\epsilon\right) = \frac{1}{2\sigma(\alpha\epsilon)}(\sigma(-\alpha \epsilon)-\sigma(\alpha (3\epsilon-T))) \leq (1-(1-\epsilon)) = \epsilon.
\end{equation}
By using a Taylor approximation (cf. \eqref{eq:rho-taylor} in the proof of Lemma \ref{lem:rho}) we find that,
\begin{equation}
    \rho_\epsilon(\epsilon)\leq \rho_\epsilon(\epsilon^3) \leq  \beta\sigma'(\beta\epsilon^3) + 20\epsilon^3 \leq 21\epsilon^3,
\end{equation}
where the second inequality follows from the definition of $\beta$ and Lemma \ref{lem:alpha-growth}. 
As a result we find from \eqref{eq:zx0} that there exists an absolute constant $C>0$ such that,
\begin{equation}
     \abs{\int_0^T z(x,0) \overline{\varphi}_\epsilon^{y,s}(x,0)ds} \leq CB\rho_\epsilon(x-y) \epsilon
\end{equation}
Combining the above with Lemma \ref{lem:rho} gives, 
\begin{equation}
\begin{split}
    \abs{\int_0^1\int_0^1\int_0^T z(x,0) \overline{\varphi}_\epsilon^{y,s}(x,0) dsdxdy} &\leq CB\epsilon \int_0^1\int_0^1 \rho_\epsilon(x-y) dx dy \leq 2CB\epsilon. 
\end{split}
\end{equation}
Similarly it holds that,
\begin{equation}
    \abs{\int_0^1\int_0^1\int_0^T z(x,T) \overline{\varphi}_\epsilon^{y,s}(x,T) dsdxdy} \leq 2CB\epsilon. 
\end{equation}
Combining the two above equation with \eqref{eq:integration-parts} and redefining $C$ then concludes the proof of the lemma. 
\end{proof}

\begin{lemma}\label{lem:partial-int-x}
Let $T,\epsilon>0$, $u\in W^{1,\infty}([0,1]\times[0,T])$ with $u(0,t)=u(1,t)$ for all $t\in[0,T]$ and $\overline{\varphi}_\epsilon^{y,s}$, $(y,s)\in [0,1]\times[0,T]$, as defined in \eqref{eq:varphi}. There exists an constant $C>0$ (independent of $u,v$ and $\epsilon$) such that,
\begin{equation}
\begin{split}
    &{\int_0^1\int_0^T\int_0^1\int_0^T\left(\overline{\varphi}_\epsilon^{y,s}(x,t) \partial_x Q[u(x,t); v(y,s)] + Q[u(x,t); v(y,s)]\partial_x \overline{\varphi}_\epsilon^{y,s}(x,t) \right)dt dx ds dy}\\
    &\qquad \leq 12L_f \int_0^T \abs{v(1,t)-v(0,t)}dt+CL_f\norm{v_x}_{\infty} (1-\ln(\epsilon))\epsilon
\end{split}
\end{equation}
\end{lemma}

\begin{proof}
Using integration by parts and Fubini's theorem we find that,
\begin{equation}
\begin{split}
    &\int_0^1\int_0^1\left(\overline{\varphi}_\epsilon^{y,s}(x,t) \partial_x Q[u(x,t); v(y,s)] + Q[u(x,t); v(y,s)]\partial_x \overline{\varphi}_\epsilon^{y,s}(x,t) \right) dx  dy\\
    &\quad = \int_0^1 \left(Q[u(1,t); v(y,s)]\overline{\varphi}_\epsilon^{y,s}(1,t)   - Q[u(0,t); v(y,s)] \overline{\varphi}_\epsilon^{y,s}(0,t) \right) dy\\
    &\quad= \int_0^1 \left(Q[u(1,t); v(1,s)]\overline{\varphi}_\epsilon^{y,s}(1,t)   - Q[u(0,t); v(0,s)] \overline{\varphi}_\epsilon^{y,s}(0,t) \right) dy \quad {\scriptstyle=:\: (A)}\\
    &\qquad+\int_0^1 \left(Q[u(1,t); v(y,s)]   - Q[u(1,t); v(1,s)] \right)\overline{\varphi}_\epsilon^{y,s}(1,t) dy \quad {\scriptstyle=:\: (B)}\\
    &\qquad+\int_0^1 \left(Q[u(0,t); v(0,s)]   - Q[u(0,t); v(y,s)] \right)\overline{\varphi}_\epsilon^{y,s}(0,t) dy \quad {\scriptstyle=:\: (C)}. 
\end{split}
\end{equation}
We first bound $(A)$. Recall that $\overline{\varphi}_\epsilon^{y,s}(x,t)=\chi_\epsilon\left(\frac{t+s}{2}\right)\rho_\epsilon(t-s)\rho_\epsilon(x-y) =: z(t,s)\rho_\epsilon(x-y)$ and that $u(0,t)=u(1,t)$. We calculate,
\begin{equation}
    \begin{split}
(A) &= z(t,s)Q[u(1,t); v(1,s)] \int_0^1 \left(\rho_\epsilon(1-y)   - \rho_\epsilon(y) \right) dy\\
&\quad + z(t,s) \left(Q[u(1,t); v(1,s)]-Q[u(0,t); v(0,s)] \right)\int_0^1 \rho_\epsilon(y) dy.
    \end{split}
\end{equation}
Using Lemma \ref{lem:Q}, Lemma \ref{lem:rho} and the fact that $u(0,t)=u(1,t)$ we find that,
\begin{equation}
    \abs{(A)}\leq z(t,s)\cdot 3L_f \abs{v(1,s)-v(0,s)}\cdot 2
\end{equation}
and as a result, 
\begin{equation}
    \abs{\int_0^T\int_0^T (A) \:dt ds} \leq 12 L_f \int_0^T \abs{v(1,s)-v(0,s)}ds, 
\end{equation}
where we used that $\int_0^T \rho_\epsilon(t-s)dt \leq 2$ (Lemma \ref{lem:rho}). 

The terms $(B)$ and $(C)$ are similar to each other and can be bounded in the same way. We will prove an upper bound on $(C)$. Using Lemma \ref{lem:Q} and the non-negativity of $\rho_\epsilon$ we find that,
\begin{equation}
    \abs{(C)}\leq 3L_f z(t,s) \int_0^1 \abs{v(y,s)-v(0,s)}\rho_\epsilon(y)dy. 
\end{equation}
In addition, using Lemma \ref{lem:smooth-abs} and Lemma \ref{lem:rho} (for any $\eta>0$), 
\begin{equation}
\begin{split}
    \int_0^{1} \abs{v(y,s)-v(0,s)}\rho_\epsilon(y)dy &\leq \int_0^{1} \abs{v(y,s)-v(0,s)}_\eta \rho_\epsilon(y)dy +C\eta\\
    &\leq 20\norm{v_x}_{\infty} (1-\ln(\epsilon))\epsilon + C\eta. 
\end{split}
\end{equation}
Finally we can easily see that $\int_0^T\int_0^T z(t,s) ds dt \leq 4$. Combining everything then proves the lemma. 
\end{proof}

\subsection{Auxiliary results for Section \ref{sec:33}}

\begin{lemma}\label{lem:bound-generalization}
Let $a,B,\mathfrak{L}\geq1$, $k,d,M\in\mathbb{N}$, $D$ a set, $(\Omega, \mathcal{A}, \mathbb{P})$ a probability space, $\Theta = [-a,a]^k$ and let $f:D\to \mathbb{R}$ and $f_\theta: D\to \mathbb{R}$ be functions for all $\theta\in\Theta$. Let $X_i:\Omega\to D$, $1\leq i \leq M$ be iid random variables, $\S=\{X_1, \ldots X_M\}$, let $\Et:\Theta\times D^M\to[0,B]$ and $\Eg:\Theta\to [0,B]$ be given by
\begin{align}\label{eqn:training-generalization-minimizer}
\begin{split}
    \Et(\theta,\S) &= \frac{1}{M}\sum_{i=1}^M\abs{f_\theta(z_i)-f(z_i)}, \quad \Eg(\theta)^2 = \int_D \abs{f_\theta(z)-f(z)} d\mu(z), 
\end{split}
\end{align}
and let $\theta^*:D^M\to\Theta$ be a function. Let $\theta\mapsto\Eg(\theta,\S)$ and $\theta\mapsto\Et(\theta)$ be Lipschitz continuous with Lipschitz constant $\mathfrak{L}$. If $M\geq e^{16/k}$ then it holds that 
\begin{equation}
    \Prob{\Eg(\theta^*(\S)) > \Et(\theta^*(\S),\S) + \sqrt{\frac{B^2k}{2M}\ln(\frac{a\mathfrak{L}\sqrt{M}}{\sqrt[k]{\delta} B})}}\leq \delta.
\end{equation}
\end{lemma}
\begin{proof}
Let $\epsilon>0$ be arbitrary, let $\{\theta_i\}_{i=1}^N$ be a $\delta$-covering of $\Theta$ with respect to the supremum norm and define the random variable $Y = \Eg(\theta^*(\S))-\Et(\theta^*(\S),\S)$.
Then if follows from equation (4.8) in the proof of \cite[Theorem 5]{deryck2021pinn} that 
\begin{equation}
    \mathbb{P}(Y > \epsilon) \leq \left(\frac{2a\mathfrak{L}}{\epsilon}\right)^k \exp\left(\frac{-2\epsilon^2M}{B^2}\right),
\end{equation}
since $\mathbb{P}(Y > \epsilon) = 1 - \Prob{\cA}$, where $\cA$ is as defined in the proof of \cite[Theorem 5]{deryck2021pinn}. 
Setting $\delta = \Prob{Y> \epsilon}$ leads to
\begin{equation}\label{eq:eps-eq}
    \epsilon = \sqrt{\frac{B^2}{2M}\ln(\frac{1}{\delta}\left(\frac{2a\mathfrak{L}}{\epsilon}\right)^k)}. 
\end{equation}
If $\delta\epsilon^{k}\leq Be^{-8}(2a\mathfrak{L})^k$ then it holds that 
\begin{equation}\label{eq:2sqrt2}
    \left[\ln(\frac{1}{\delta}\left(\frac{2a\mathfrak{L}}{\epsilon}\right)^k)\right]^{-1/2} \leq \frac{1}{2\sqrt{2}}. 
\end{equation}
Using \eqref{eq:eps-eq} and \eqref{eq:2sqrt2} gives us, 
\begin{equation}
    \begin{split}
\epsilon &\leq \sqrt{\frac{B^2}{2M}\ln(\frac{(2a\mathfrak{L})^{k}}{\delta}\left(\frac{\sqrt{2M}}{B}\left[\ln(\frac{1}{\delta}\left(\frac{2a\mathfrak{L}}{\epsilon}\right)^k)\right]^{-1/2}\right)^{k})}\\
&\leq \sqrt{\frac{B^2}{2M}\ln(\frac{(a\mathfrak{L}\sqrt{M})^{k}}{\delta B^k})} = \sqrt{\frac{B^2k}{2M}\ln(\frac{a\mathfrak{L}\sqrt{M}}{\sqrt[k]{\delta} B})}.
    \end{split}
\end{equation}
Finally, we can use \eqref{eq:eps-eq} to observe that the condition $\delta\epsilon^{k}\leq Be^{-8}(2a\mathfrak{L})^k$ is met if $\delta<1\leq B$ and $M\geq e^{16/k}$ because
\begin{equation}
    \left(\frac{e^{8/k}}{2a\cL}\right)^2 \frac{2}{B^2\ln((2a\cL/\epsilon)^k/\delta)} \leq e^{16/k}. 
\end{equation}
\end{proof}

\end{document}